\documentclass[11pt]{amsart}
\usepackage{mesfonts}

\def\S{\mathcal S}
\def\M{\mathcal{M}}
\def\F{\mathcal{F}}

\def\p{\mathbb{P}}

\def\Z{\mathcal{Z}}

\def\hmu{\hat\mu}
\def\m{m^{*}}
\def\hd{\hat \delta_{\beta}}
\def\D{\mathcal{D}}
\def\EL{\mathcal{E}}
 
\def\dim{\textrm{dim}}
\def\span{\textrm{span}}
\RequirePackage{amssymb}

\newtheorem{lemma}{Lemma}
\newtheorem{theorem}{Theorem}
\newtheorem{proposition}{Proposition}
\newtheorem{corollary}{Corollary}

\usepackage{graphicx}

\RequirePackage{amssymb}

\begin{document}
\title{Mixing Least-Squares Estimators when the Variance is Unknown }
\author{Christophe Giraud}
\address{Universit\'e de Nice Sophia-Antipolis, Laboratoire J-A Dieudonn\'e, Parc Valrose, 06108 Nice cedex 02}
\email{cgiraud@math.unice.fr}
\date{First draft 03/02/2007, Revision 02/11/2007}

\begin{abstract}
We propose a procedure to handle the problem of Gaussian regression when the variance is unknown. We mix least-squares estimators from various models according to a procedure inspired by that of Leung and Barron \cite{LB06}. 
We  show that in some cases the resulting estimator is a simple shrinkage estimator. We then apply this procedure in various statistical settings such as linear regression or adaptive estimation in Besov spaces.
Our results provide non-asymptotic risk bounds for the Euclidean risk of the estimator.
\end{abstract}

\maketitle

%\tableofcontents

\section{Introduction}

We consider the regression framework, where we have  noisy observations
\begin{equation}
  \label{modele}
 Y_{i}=\mu_i+\sigma\eps_i,\ \ i=1,\ldots,n 
\end{equation}
of an unknown vector $\mu=\pa{\mu_{1},\ldots,\mu_{n}}'\in\R^{n}$. We assume  that the $\eps_{i}$'s are i.i.d standard Gaussian random variables and that the noise level $\sigma>0$ is unknown. Our aim is to estimate~$\mu$.

In this direction, we introduce a finite collection $\ac{\S_{m}, \ m\in \M}$ of linear spaces  of $\R^n$, which we call henceforth models. To each model $\S_{m}$, we  associate the least-squares estimator $\hmu_{m}=\Pi_{\S_{m}}Y$ of $\mu$ on $\S_{m}$, where $\Pi_{\S_{m}}$ denotes the orthogonal projector onto $\S_{m}$. The $L^2$-risk of the estimator $\hmu_{m}$ with respect to the Euclidean norm $\|\cdot\|$ on $\R^n$
is 
\begin{equation}\label{risk1}
\E\cro{\|\mu-\hat \mu_{m}\|^{2}}= \|\mu-\Pi_{\S_{m}}\mu\|^{2}+\dim(\S_{m})\sigma^{2}.
\end{equation}

Two strategies have emerged to handle the problem of the choice of an estimator of $\mu$ in this setting. One strategy is to select a model $\S_{\hat m}$ with a data driven criterion  and use $\hmu_{\hat m}$ to estimate $\mu$. In the favorable cases, the risk of this estimator is of order the minimum over $\M$ of the risks \eref{risk1}. Model selection procedures have received a lot of attention in the literature, 
starting from the pioneer work of Akaike \cite{A69} and Mallows \cite{M73}. 
It is beyond the scope of this paper to make an historical review of the topic. We simply mention in the Gaussian setting the papers of Birg\'e and Massart \cite{BM01, BM06} (influenced by Barron and Cover \cite{BC91} and Barron, Birg\'e and Massart \cite{BBM99})  which give  non-asymptotic risk bounds for a selection criterion generalizing  Mallows'$C_{p}$.

An alternative  to model selection is mixing. One estimates $\mu$ by a convex (or linear) combination of the $\hmu_{m}$s
\begin{equation}\label{mixing}
\hat \mu=\sum_{m\in \M}w_{m}\hat\mu_{m},
\end{equation}
with weights $w_{m}$ which are $\sigma(Y)$-measurable random variables.
This strategy is not suitable when the goal is to select a single model $\S_{\hat m}$, nevertheless it enjoys the nice property that $\hmu$ may perform better than the best of the $\hmu_{m}$s. 
Various choices of weights $w_{m}$ have been proposed, from an  information theoretic or Bayesian  perspective. Risk bounds have been provided 
 by Catoni \cite{C99}, Yang \cite{Y00, Y04}, Tsybakov~\cite{T03} and Bunea {\it et al.}~\cite{BTW07}  for regression on a random design and by
 Barron \cite{B87}, Catoni \cite{C97} and Yang \cite{Y00b} for density estimation. 
 For the Gaussian regression framework we consider here, Leung and Barron \cite{LB06} propose a mixing procedure  for which they  derive a precise non-asymptotic  risk bound. When the collection of models is not too complex, this bound shows that the risk of their estimator $\hmu$ is close to the minimum over $\M$ of the risks \eref{risk1}. Another nice feature of their mixing procedure is that both the weights $w_{m}$ and the estimators $\hat\mu_{m}$ are build on the same data set, which enable to handle cases where the law of the data set is not exchangeable. Unfortunately, their choice of  weights $w_{m}$ depends on the variance $\sigma^2$, which is usually unknown.
  
 In the present paper, we consider the more practical situation where the variance $\sigma^2$ is unknown. Our mixing strategy is akin  to that of Leung and Barron \cite{LB06}, but is not depending on the variance $\sigma^2$. In addition, we show that both our estimator and the estimator of Leung and Barron are simple shrinkage estimators in some cases. From a  theoretical point of view, we relate our weights $w_{m}$ to a Gibbs measure on $\M$ and  derive a sharp risk bound for the  estimator $\hmu$. Roughly, this bound says  that the risk of $\hmu$ is  close to the minimum over $\M$ of the risks \eref{risk1} in the favorable cases. We then discuss  the  choice of the collection of models $\ac{\S_{m}, m\in\M}$ in various situations. Among others, we 
   produce an   estimation procedure which is adaptive over a large class of Besov balls.

 Before presenting our mixing procedure, we briefly recall that of Leung and Barron \cite{LB06}. Assuming that
 the variance $\sigma^2$ is known, they use  the  weights
 \begin{equation}\label{LB}
w_{m}={\pi_{m}\over \Z}\exp\pa{-{\beta}\cro{\|Y-\hmu_{m}\|^2/\sigma^2+2\dim(\S_{m})-n}},\quad m\in\M
\end{equation}
where $\ac{\pi_{m},\ m\in\M}$ is a given prior distribution on $\M$  and $\Z$ normalizes the sum of   the $w_{m}$s to one. 
These weights have a Bayesian flavor. Indeed, they appear with $\beta=1/2$ in Hartigan \cite{H02}  which considers the Bayes procedure with the following (improper) prior distribution: pick an $m$ in $\M$ according to $\pi_{m}$ and then sample $\mu$ "uniformly" on $\S_{m}$. Nevertheless,  in Leung and Barron \cite{LB06} the role of the prior distribution $\ac{\pi_{m},\ m\in\M}$ is to  favor models with low complexity. Therefore,
the choice of $\pi_{m}$ is driven by the complexity of the model $\S_{m}$  rather than from a prior knowledge on $\mu$. In this sense  their approach  differs from the classical Bayesian point of view. Note that the term $\|Y-\hmu_{m}\|^2/\sigma^2+2\dim(\S_{m})-n$ appearing in the weights \eref{LB} is an unbiased estimator of the risk \eref{risk1} rescaled by $\sigma^2$.  
The  size of the weight $w_{m}$ then depends on  the difference between  this estimator of the risk \eref{risk1} and  $-\log(\pi_{m})$, which can be thought as a  complexity-driven penalty (in the spirit of Barron and Cover~\cite{BC91} or Barron {\it et al.}~\cite{BBM99}). The parameter $\beta$ tunes the balance between this two terms.
 For  $\beta\leq 1/4$, Theorem 5 in \cite{LB06} provides a sharp risk bound for the procedure.

The rest of the  paper is organized as follows. We present our mixing strategy in the next section and express in some cases  the resulting estimator $\hmu$ as a shrinkage estimator. In Section~3, we state non-asymptotic risk bounds for the
procedure and discuss  the choice of the tuning parameters.
Finally, we propose in Section~4 some weighting strategies for linear regression or for adaptive regression over Besov balls. Section~5 is devoted to 
a numerical illustration and Section~6 to
the proofs. Additional results are given in the Appendix.

We end this section with some notations we shall use along this paper. We write $|m|$ for the cardinality of a finite set $m$, and $<x,y>$ for the inner product of two vectors $x$ and $y$ in $\R^n$. To any real number $x$, we denote by $(x)_{+}$  its positive part and by $\lfloor x \rfloor$ its integer part.

\section{The estimation procedure}
We assume henceforth that $n\geq 3$.
\subsection{The estimator}
We start with a finite collection of models  $\ac{\S_{m},\;m\in\M}$ and to  each model $\S_{m}$ we associate the least-squares estimator $\hmu_{m}=\Pi_{\S_{m}}Y$ of $\mu$ on $\S_{m}$. We also introduce  
 a probability distribution  $\ac{\pi_{m},\ m\in\M}$ on $\M$, which is meant to take into account the complexity of the family and favor models with low dimension. For example, if the collection $\ac{\S_{m},\;m\in\M}$ has (at most) $e^{ad}$ models per dimension $d$, we suggest to choose $\pi_{m}\propto e^{(a+1/2)\dim(\S_{m})}$, see the example at the end of Section~\ref{generalbound}. As mentioned before, the quantity $-\log(\pi_{m})$ can be interpreted as a complexity-driven penalty associated to the model $\S_{m}$ (in the sense of Barron {\it et al.}~\cite{BBM99}). The performance of our estimation procedure depends strongly on the choice of the collection of models $\ac{\S_{m},\;m\in\M}$ and the probability distribution  $\ac{\pi_{m},\;m\in\M}$. We detail  in Section~\ref{exemples} some suitable choices of these families for linear regression and  estimation of BV or Besov functions.

Hereafter, we assume  that there exists a linear space $\S_{*}\subset \R^n$ of dimension $d_{*}<n$, such that $\S_{m}\subset \S_{*}$ for all $m\in \M$. We will take advantage of this situation and estimate the variance of the noise by 
\begin{equation}\label{shat}
\hat\sigma^2={\|Y-\Pi_{\S_{*}}Y\|^2\over N_{*}}
\end{equation}
where $N_{*}=n-d_{*}$. We emphasize that we do not assume that $\mu\in\S_{*}$ and the estimator $\hat\sigma^2$ is (positively) biased in general. It turns out that our estimation procedure does not need a precise estimation of the variance $\sigma^2$ and the choice~\eref{shat}  gives good results. In practice, we  may replace  the residual estimator $\hat\sigma^2$ by 
a difference-based estimator (Rice~\cite{R84}, Hall {\it et al.}~\cite{HKT90}, Munk {\it et al.}~\cite{MBWF05}, Tong and Wang~\cite{TW05}, Wang {\it et al.}~\cite{WBCL07}, etc) or by any non-parametric estimator  (e.g.\;Lenth~\cite{L89}), but we are not able to prove any bound similar to~\eref{fgibbs} or~\eref{foracle} when using one of these estimators.

Finally, we associate to the collection of models  $\ac{\S_{m},\;m\in\M}$, a collection  $\ac{L_{m},\ m\in\M}$ of non-negative weights. We recommend to set $L_{m}=\dim(\S_{m})/2$, but any (sharp) upper bound of this quantity may also be appropriate, see the discussion after  Theorem~\ref{resultat}.
Then, for a given  positive constant $\beta$   we define the estimator $\hmu$  by
\begin{equation}\label{estimateur}
\hat \mu=\sum_{m\in \M}w_{m}\hat\mu_{m}\quad\textrm{with}\quad w_{m}={\pi_{m}\over \Z}\exp\pa{\beta{\|\hmu_{m}\|^2\over \hat \sigma^2}-{L_{m}}},
\end{equation}
 where  $\Z$ is a constant that normalizes the sum of the $w_{m}$s to one. An alternative formula for $w_{m}$ is $w_{m}=\pi_{m}\exp\pa{-\beta{\|\Pi_{\S_{*}}Y-\hmu_{m}\|^2/\hat \sigma^2}-{L_{m}}}/\Z'$
with $\Z'=e^{-\beta\|\Pi_{\S_{*}}Y\|^2/\hat \sigma^2}\Z$.
We can interpret the term $\|\Pi_{\S_{*}}Y-\hmu_{m}\|^2/ \hat \sigma^2+L_{m}/\beta$ appearing in the exponential  as a (biased) estimate of the risk \eref{risk1} rescaled by $\sigma^2$. As in \eref{LB}, the balance in the weight $w_{m}$ between this estimate of the risk and the penalty $-\log(\pi_{m})$ is tuned by $\beta$.
We refer to the discussion after  Theorem~\ref{resultat}
for the choice of this parameter. We mention that 
the weights $\ac{w_{m},\, m\in\M}$ can be viewed as a Gibbs measure on $\M$ and we will  use this property to assess the performance of the procedure.

We emphasize in the next section, that $\hmu$ is a simple shrinkage estimator in some cases.

\subsection{A simple shrinkage estimator}\label{casortho}
In this section, we focus on the case where $\M$ consists of all the subsets of $\ac{1,\ldots,p}$, for some $p<n$ and 
$\S_{m}=\textrm{span}\ac{v_{j}, \ j\in m}$
with $\ac{v_{1},\ldots,v_{p}}$ an orthonormal family of vectors in $\R^n$. We use the convention 
$\S_{\emptyset}=\ac{0}$. An example of such a setting is given in Section \ref{exhaar}, see also the numerical illustration Section~\ref{illustration}. Note that $\S_*$ corresponds here to $\S_{\ac{1,\ldots,p}}$ and $d_*=p$.

 To favor models with small dimensions,  we choose the probability distribution 
\begin{equation}\label{orthopoids}
\pi_{m}=\pa{1+{1\over p^{\alpha}}}^{-p}p^{-\alpha |m|},\quad m\in\M,
\end{equation}
with $\alpha>0$. We also set $L_{m}=b |m|$ for some $b\geq 0$.
\begin{proposition}\label{fast}
Under the above assumptions, we have the following expression for $\hmu$
\begin{equation}\label{fasteq}
\hmu=\sum_{j=1}^{p}(c_{j}Z_{j})v_{j},\quad \textrm{with}\quad  Z_{j}=<Y,v_{j}>\textrm{and}\quad c_{j}={\exp\pa{\beta Z_{j}^2/{\hat \sigma^2}}\over p^\alpha\exp\pa{b}+\exp\pa{\beta Z_{j}^2/{\hat \sigma^2}}}\,.
\end{equation}
\end{proposition}
The proof of this proposition is postponed to Section \ref{preuvefast}. 
The main interest of Formula~\eref{fasteq} is to allow a fast computation 
of $\hmu$. Indeed, we only need to compute the $p$ coefficients $c_{j}$ instead of the
 $2^{p}$ weights $w_{m}$ of  formula \eref{estimateur}.
% We also use Formula \eref{fasteq} in Section \ref{perfortho} to analyse the performance of $\hmu$ for values of $\beta$ not covered by Theorem \ref{resultat}.

The coefficients $c_{j}$ are shrinkage coefficients taking values in $[0,1]$. They are   close to one when $Z_{j}$ is large and close to zero when  $Z_{j}$ is small. The transition from 0 to 1 occurs when 
$ Z_{j}^2\approx\beta^{-1} (b+\alpha\log p)\hat \sigma^2$. The choice of the tuning parameters $\alpha$, $\beta$ and $b$ will be discussed in Section~\ref{parametresortho}.

{\bf Remark 1.}  Other choices are possible for $\ac{\pi_{m},\ m\in\M}$ and they lead to different $c_{j}$s. Let us mention  the choice  $\pi_{m}=\pa{(p+1)\binom{p}{|m|}}^{-1}$ for which the $c_{j}$s are given by
$$c_{j}={\int_{0}^1q\prod_{k\neq j}\cro{q+(1-q)\exp\pa{-\beta Z_{k}^2/\hat \sigma^2+b}}dq \over \int_{0}^1\prod_{k=1}^p\cro{q+(1-q)\exp\pa{-\beta Z_{k}^2/\hat \sigma^2+b}}dq},\quad\ \textrm{for }j=1,\ldots,p. $$
This formula can be derived from the Appendix of Leung and Barron \cite{LB06}.

\vspace{0.2cm}
{\bf Remark 2.} 
When the variance is known, we can give a formula similar to \eref{fasteq} for the estimator of Leung and Barron \cite{LB06}. Let us consider the same setting, with $p\leq n$.  Then, when the distribution $\ac{\pi_{m},\ m\in\M}$ 	is given by \eref{orthopoids}, the estimator \eref{mixing} with weights $w_{m}$ given by \eref{LB} takes the form
\begin{equation}\label{fLB}
\hmu=\sum_{j=1}^p \pa{{e^{\beta Z_{j}^2/\sigma^2}\over p^{\alpha}e^{2\beta }+e^{\beta Z_{j}^2/\sigma^2}}\,Z_{j}}v_{j}.
\end{equation}

\section{The performance}\label{perf}
\subsection{A general risk bound}\label{generalbound}
The next result gives an upper bound on the $L^2$-risk of the estimation procedure. We remind the reader that $n\geq 3$ and set 
\begin{equation}\label{phi}
\begin{array}{lcll}
\phi :& ]0,1[ & \to & ]0,+\infty[\\
& x &\mapsto &(x-1-\log x)/2\ ,
\end{array}
\end{equation}
which is  decreasing.
\begin{theorem}\label{resultat}
Assume  that $\beta$ and $N_{*}$ fulfill the condition
\begin{equation}\label{conditions}
\beta< 1/4 \quad \textrm{and}\quad N_{*}\geq 2+{\log n\over \phi(4\beta)},
\end{equation}
with $\phi$ defined by \eref{phi}.
Assume also that $L_{m}\geq \textrm{dim}(\S_{m})/2$, for all $m\in\M$. Then, we have  the following upper bounds on the $L^2$-risk of the estimator $\hmu$ 
\begin{eqnarray}
\lefteqn{\E\pa{{\|\mu-\hmu\|^ {2}}}}\nonumber\\
&\leq& -(1+\varepsilon_{n})\ {\bar\sigma^2\over \beta}\ \log\cro{\sum_{m\in\M}\pi_{m}e^{-{\beta}\cro{{\|\mu-\Pi_{\S_{m}}\mu\|^{2}}-\dim(\S_{m})\sigma^2}/\bar\sigma^2-L_{m}}}+{\sigma^2\over 2\log n}\label{fgibbs}\\
&\leq& \pa{1+\varepsilon_{n}}\inf_{m\in\M}\ac{{{\|\mu-\Pi_{\S_{m}}\mu\|^2}
%-\textrm{dim}(\S_{m})\sigma^2
+{\bar\sigma^2\over \beta}\pa{L_m-\log\pi_m}}}+{\sigma^2\over 2\log n},\label{foracle}
\end{eqnarray}
where $\varepsilon_{n}=(2n\log n)^{-1}$ and $\bar\sigma^2=\sigma^2+\|\mu-\Pi_{\S_{*}}\mu\|^2/N_*$.
\end{theorem}
The proof  Theorem \ref{resultat} is delayed to Section \ref{preuveresultat}. Let us  comment this result.

To start with,  the Bound \eref{fgibbs} may look somewhat cumberstone but it improves \eref{foracle} when there are several good models to estimate $\mu$. For example, we can derive from \eref{fgibbs} the bound
\begin{multline*}
\E\pa{{\|\mu-\hmu\|^ {2}}}\leq  \\
\pa{1+\varepsilon_{n}}\inf_{m\in\M}\ac{{{\|\mu-\Pi_{\S_{m}}\mu\|^2}
+{\bar\sigma^2\over \beta}\pa{L_m-\log\pi_m}}}+
\inf_{\delta\geq 0}\ac{\delta-{\bar\sigma^2\over \beta}\log|\M_{\delta}|}+{\sigma^2\over 2\log n},
\end{multline*}
where
$\M_{\delta}$ is the set made of those $m^*$ in $\M$ fulfilling
$$ {{\|\mu-\Pi_{\S_{m^*}}\mu\|^2}
+{\bar\sigma^2\over \beta}\pa{L_{m^*}-\log\pi_{m^*}}}\leq \delta+\inf_{m\in\M}\ac{{{\|\mu-\Pi_{\S_{m}}\mu\|^2}
+{\bar\sigma^2\over \beta}\pa{L_m-\log\pi_m}}}.$$
In the extreme case where all the quantities
${{\|\mu-\Pi_{\S_{m}}\mu\|^2}+{\bar\sigma^2\over \beta}\pa{L_m-\log\pi_m}}$ 
are equal,  \eref{fgibbs} then improves \eref{foracle}   by a factor  $\beta^{-1}{\bar\sigma^2}\log|\M|$.

We now discuss the choice of the parameter $\beta$ and the weights $\{L_{m},\ m\in\M\}$.
% once the priors $\{\pi_{m},\ m\in\M\}$ have been choosen.
The choice $L_{m}={\rm dim}(\S_{m})/2$ seems to be the more accurate  since it satisfies the conditions of Theorem~\ref{resultat} and minimizes the right hand side of~\eref{fgibbs} and~\eref{foracle}. We shall mostly use this one in the following, but there are some cases where it is easier to use some (sharp) upper bound of the dimension of $\S_{m}$ instead of dim$(\S_{m})$ itself, see for example Section \ref{besov}.

The largest parameter $\beta$ fulfilling Condition \eref{conditions} is
\begin{equation}\label{betaopt}
\beta={1\over 4}\ \phi^{-1}\pa{\log n \over N_{*}-2}<{1\over 4}.
\end{equation}
We suggest to use this value, since it minimizes the right hand side of~\eref{fgibbs} and~\eref{foracle}. 
%Note that when $N_{*}$ is large compared to $4\log n$, 
%$$\phi^{-1}\pa{\log n \over N_{*}-2}\sim 1-\sqrt{4\log n \over N_{*}-2}.$$
Nevertheless, as discussed in Section~\ref{parametresortho} 
for the situation of Section \ref{casortho}, it is sometimes possible to use larger values for $\beta$.

Finally, we would like to compare the bounds of Theorem \ref{resultat} with the minimum over $\M$ of the risks given by \eref{risk1}. Roughly, the Bound~\eref{foracle} states that the estimator $\hmu$ achieves the best trade-off between the bias $\|\mu-\Pi_{\S_{m}}\mu\|^2/\sigma^2$ and the complexity term $C_{m}=L_{m}-\log \pi_{m}$. More precisely, we derive from~\eref{foracle} the (cruder) bound
\begin{equation}\label{bias-complexity}
\E\pa{{\|\mu-\hmu\|^ {2}}}\leq \pa{1+\varepsilon_{n}}\inf_{m\in\M}\ac{{{\|\mu-\Pi_{\S_{m}}\mu\|^2}+{1\over \beta}C_{m}\sigma^2}}+R^*_{n}\sigma^2,
\end{equation}
with 
$$\varepsilon_{n}={1\over 2n\log n}\qquad\textrm{and}\qquad R^*_{n}={1\over 2\log n}+{\|\mu-\Pi_{\S_{*}}\mu\|^2\over \beta N^*\sigma^2}\,\sup_{m\in\M}C_{m}.$$
In particular, if $C_{m}$ is of order $\dim(\S_{m})$, then~\eref{bias-complexity} allows to compare the risk of $\hmu$ with the infimum of the risks~\eref{risk1}.
We discuss this point in the following example.

{\bf Example.} Assume that the family $\M$ has an index of complexity $(M,a)$, as defined in \cite{BGH06}, which means that
$$|\ac{m\in \M, \ \dim(\S_{m})=d}|\leq M e^{ad},\quad\textrm{for all } d\geq 1.$$
If we choose, example given,
\begin{equation}\label{primaire}
\pi_{m}={e^{-(a+1/2)\dim(\S_{m})}\over \sum_{m'\in\M}e^{-(a+1/2)\dim(\S_{m'})}}\quad\textrm{and}\quad L_{m}=\dim(\S_{m})/2,
\end{equation}
   we have
$C_{m}\leq (a+1)\dim(\S_{m})+ \log\pa{3M}$.
Therefore, when $\beta$ is given by~\eref{betaopt}  and $d_{*}\leq \kappa n$ for some $\kappa<1$, we have 
$$\E\pa{{\|\mu-\hmu\|^ {2}}}\leq \pa{1+\varepsilon_{n}}\inf_{m\in\M}\ac{{{\|\mu-\Pi_{\S_{m}}\mu\|^2}+{a+1\over \beta}\dim(\S_{m})\sigma^2}}+R'_{n}\sigma^2,$$
with
$$R'_{n}={\log(3M)\over \beta}+{1\over 2\log n}+{\|\mu-\Pi_{\S_{*}}\mu\|^2\over \sigma^2}\times {(a+1)\kappa+n^{-1}\log(3M)\over \beta(1-\kappa)}\,.$$
In particular, for a given index of complexity $(M,a)$ and a given  $\kappa$, the previous bound gives an oracle inequality.

\subsection{On the choice of the parameter $\beta$.}\label{parametresortho}
The choice of the tuning parameter $\beta$ is  important in practice. 
Theorem~\ref{resultat} or Theorem~5 in \cite{LB06} justify the choice of a  $\beta$ smaller than $1/4$. Nevertheless, Bayesian arguments~\cite{H02} suggest to take a larger value for $\beta$, namely $\beta=1/2$. In this section, we discuss this issue on the example of Section~\ref{casortho}.

For the sake of simplicity, we will restrict to  the case where the variance is known. 
We consider the weights~\eref{LB} proposed by Leung and Barron, with the probability distribution $\pi$  given by~\eref{orthopoids} with  $\alpha=1$, namely\footnote{Note that this choice of $\alpha$ minimizes the rate of growth of $-\log \pi_{m}$ when $p$ goes to infinity since
$$-\log\pi_{m}=p\log\pa{1+{p^{-\alpha}}}+\alpha|m|\log p= p^{1-\alpha}+\alpha|m|\log p+o\pa{p^{1-\alpha}},$$
 when $\pi_{m}$ is given by \eref{orthopoids} with $\alpha>0$.} $$\pi_{m}=\pa{1+{p^{-1}}}^{-p}p^{-|m|}.$$
According to \eref{fLB}, the estimator $\hmu$  takes the form 
\begin{equation}\label{LBortho}
\hmu=\sum_{j=1}^p{s_{\beta}(Z_{j}/\sigma)Z_{j}}\,v_{j}\quad\textrm{with}\;Z_{j}=<Y,v_{j}>\quad\textrm{and}\;  s_{\beta}(z)={e^{\beta z^2}\over pe^{2\beta }+e^{\beta z^2}}.
\end{equation}

To start with, we note that a choice $\beta>1/2$ is not to be recommanded. Indeed, we  can compare the shrinkage coefficient $s_{\beta}(Z_{j}/\sigma)$  to a threshold at level $T=(2+\beta^{-1}\log p)\sigma^2$ since
$$s_{\beta}(Z_{j}/\sigma)\geq {1\over 2}\,{\bf 1}_{\ac{Z_{j}^2\geq T}}.$$
For $\mu=0$, the risk of $\hmu$ is then larger than a quarter of the risk of the threshold estimator $\hmu_{T}=\sum_{j=1}^p{\bf 1}_{\ac{Z_{j}^2\geq T}}Z_{j}\,v_{j}$, namely
$$\E\pa{\|0-\hmu\|^2}=\sum_{j=1}^p\E\pa{s_{\beta}(Z_{j}/\sigma)^2Z_{j}^2}\geq {1\over 4}\sum_{j=1}^p\E\pa{{\bf 1}_{\ac{Z_{j}^2\geq T}}Z_{j}^2}={1\over 4}\E\pa{\|0-\hmu_{T}\|^2}.$$
Now, when the threshold $T$ is of order $2K\log p$ with $K<1$,
 the threshold estimator is known to behave poorly for $\mu=0$, see \cite{BM01} Section 7.2.
Therefore, a choice  $\beta>1/2$ would give poor results at least  when $\mu=0$.

\begin{figure}
\centerline{\includegraphics[width=10cm]{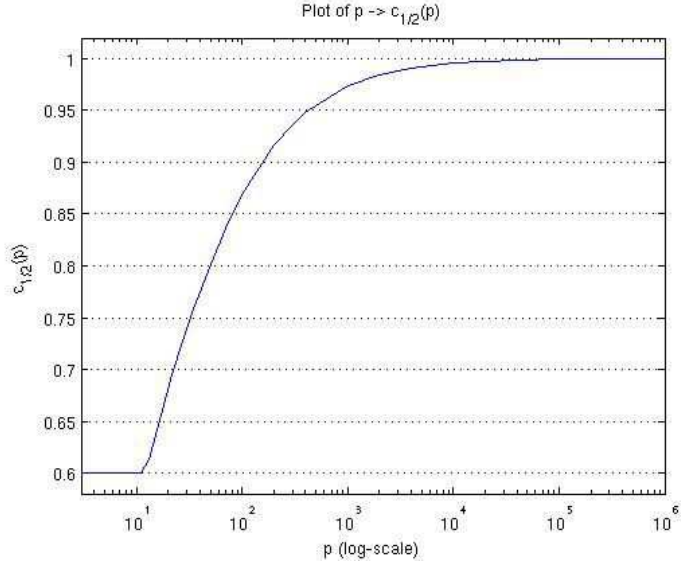}}
\caption{\label{plot} Plot of $p\mapsto c_{1/2}(p)$ }
\end{figure}
On the other hand, next proposition  justifies the use of any  $\beta\leq 1/2$ by a risk bound  similar to~\eref{foracle}.  For $p\geq 1$ and $\beta>0$, we introduce the numerical constants $\gamma_{\beta}(p)=\sqrt{2+\beta^{-1}\log p}$ and 
$$c_{\beta}(p)= \sup_{x\in\cro{0,4\gamma_{\beta}(p)}}\cro{\int_{\R}\pa{x-(x+z)s_{\beta}(x+z)}^2e^{-z^2/2}\,{dz/ \sqrt{2\pi}}\over \min(x^2,\gamma_{\beta}(p)^2)+\gamma_{\beta}(p)^2/p}\vee 0.6.$$
This constant $c_{\beta}(p)$ can be numerically computed.  For example, $c_{1/2}(p)\leq 1$ for any $3\leq p\leq 10^6$,
 see Figure~\ref{plot}.
\begin{proposition}\label{resultatvarianceconnue}
For $3\leq p\leq n$ and $\beta\in[1/4,1/2]$,
 the Euclidean risk of  the estimator~\eref{LBortho} is 
 upper bounded by
\begin{equation}\label{formulevarianceconnue}
\E\pa{\|\mu-\hmu\|^2}\leq\|\mu-\Pi_{\S_{*}}\mu\|^2+ c_{\beta}(p)\inf_{m\in\M}\cro{{\|\Pi_{\S_{*}}\mu-\Pi_{\S_{m}}\mu\|^2}+\pa{2+\beta^{-1}\log p}(|m|+1)\sigma^2}.
\end{equation}
The constant $c_{\beta}(p)$ is (crudely) bounded by 16 when $p\geq 3$ and  $\beta\in[1/4,1/2]$.
\end{proposition}
We delayed the proof of Proposition \ref{resultatvarianceconnue} to Section~\ref{preuveprop2}. We also emphasize that the bound  $c_{\beta}(p)\leq 16$ is crude.

In light of the above result, $\beta=1/2$ seems to be a good choice  in this case
and corresponds to the choice of Hartigan~\cite{H02}. Note that the choice $\beta=1/2$ has no reason to be a good choice in other situations. Indeed, a different choice of $\alpha$ in~\eref{orthopoids} would give different "good" values for $\beta$.  For $\alpha\geq 1$, one may check that the "good" values for $\beta$ are $\beta\leq \alpha/2$.

{\bf Remark 3.}
Proposition \ref{resultatvarianceconnue} does not provide an oracle inequality.
The Bound~\eref{formulevarianceconnue} differs from the best trade off between the bias and the variance term by a $\log p$ factor. This is unavoidable from a minimax point of view as noticed in Donoho and Johnstone \cite{DJ94}. 

{\bf Remark 4.}
A similar analysis can be done for the estimator~\eref{fasteq} when the variance is unknown. When the parameters $\alpha$ and $b$ in~\eref{fasteq} equal 1, we can justify the use of values of $\beta$ fulfilling
\begin{equation*}\label{betaortho}
\beta\leq {1\over 2}\ \phi^{-1}\pa{\log p \over n-p},\qquad \textrm{for } n>p\geq 3,
\end{equation*}
see the Appendix.

\section{Choice of the models and the weights in different settings}\label{exemples}
In this section, we propose some choices of weights in three situations: the linear regression, the estimation of functions with bounded variation and regression in Besov spaces.

\subsection{Linear regression}
We consider  the case where the signal $\mu$ depends linearly on some observed explanatory variables $x^{(1)},\ldots,x^{(p)}$, namely
$$\mu_{i}=\sum_{j=1}^p\theta_{j}x_{i}^{(j)},\quad i=1,\ldots, n.$$
The index $i$ usually corresponds to an index of the experiment or to a time of observation. The number $p$ of variables may be large, but we assume here that $p$ is bounded by $n-3$.

\subsubsection{The case of ordered variables}
In some favorable situations, the explanatory variables $x^{(1)},\ldots,x^{(p)}$ are naturally ordered. In this case, we will consider the models spanned by the $m$ first explanatory variables, with $m$ ranging from 0 to $p$. In this direction, we set $\S_{0}=\ac{0}$ and  $\S_{m}=\span\ac{x^{(1)},\ldots,x^{(m)}}$ for $m\in\ac{1,\ldots,p}$, where $x^{(j)}=(x^{(j)}_{1},\ldots,x^{(j)}_{n})'$. This collection of models is  indexed by $\M=\ac{0,\ldots,p}$ and contains one model per dimension.
Note that  $\S_{*}$ coincides here with $\S_{p}$.

We may use in this case the priors
$$\pi_{m}={e^{\alpha}-1\over e^{\alpha}-e^{-\alpha p}}\ e^{-\alpha m},\quad m=0,\ldots,p,$$
with $\alpha>0$, set $L_{m}=m/2$ and takes the value \eref{betaopt} for $\beta$, with $N_{*}=n-p$. Then, according to Theorem~\ref{resultat} the performance of our procedure is controlled by
\begin{eqnarray*}
\lefteqn{\E\pa{{\|\mu-\hmu\|^ {2}}}}\\
&\leq& \pa{1+\varepsilon_{n}}\inf_{m\in\M}\ac{{{\|\mu-\Pi_{\S_{m}}\mu\|^2}+{\bar\sigma^2\over \beta}\,{(\alpha+1/2)m}}}+1.2\,{\bar\sigma^2\over \beta}\,\log\pa{e^{\alpha}\over e^{\alpha}-1}+{\sigma^2\over 2\log n}.
\end{eqnarray*}
with $\varepsilon_{n}=(2n\log n)^{-1}$ and $\bar\sigma^2=\sigma^2+\|\mu-\Pi_{\S_{p}}\mu\|^2/(n-p)$.
As mentioned at the end of Section~\ref{generalbound}, the previous bound can be formulated as an oracle inequality  when imposing the condition $p\leq \kappa n$, for some $\kappa<1$.

\subsubsection{The case of unordered variables}
When there is no natural order on the explanatory variables, we have to consider a larger collection of models. Set some $q\leq p$
which represents the maximal number of explanatory variables we want to take into account. Then, we write $\M$ for all the subsets of $\ac{1,\ldots,p}$ of size less than $q$ and $\S_{m}=\span\ac{x^{(j)},\ j\in m}$ for any nonempty $m$. We also set $\S_{\emptyset}=\ac{0}$ and $\S_{*}=\span\ac{x^{(1)},\ldots,x^{(p)}}$. Note that the cardinality of $\M$ is of order $p^q$, so when $p$ is large the value $q$ should remain small in practice.

A possible choice for $\pi_{m}$ is
$$\pi_{m}=\cro{\binom{p}{|m|}(|m|+1)H_{q}}^{-1}\quad\textrm{with}\quad
H_{q}=\sum_{d=0}^q{1\over d+1}\ \leq\ 1+\log(q+1).$$
Again, we choose the value \eref{betaopt} for $\beta$ with $N_{*}=n-p$ and $L_{m}=|m|/2$. With this choice, 
combining the inequality $\binom{p}{|m|}\leq (e|m|/p)^{|m|}$ with Theorem \ref{resultat}
gives the following bound on the risk of the procedure
\begin{eqnarray*}
\lefteqn{\E\pa{{\|\mu-\hmu\|^ {2}}}}\\
&\leq& \cro{1+\varepsilon_{n}}\inf_{m\in\M}\ac{{{\|\mu-\Pi_{\S_{m}}\mu\|^2}+{\bar\sigma^2\over \beta}\cro{|m|\pa{3/2+\log{p\over |m|}}+\log(|m|+1)}}}\\
& &+{\sigma^2\over 2\log n}+{\bar\sigma^2\over \beta}\log\log [(q+1)e]\,,
\end{eqnarray*}
with $\varepsilon_{n}=(2n\log n)^{-1}$ and $\bar\sigma^2=\sigma^2+\|\mu-\Pi_{\S_{*}}\mu\|^2/(n-p)$.

{\bf Remark.} When the family  $\ac{x^{(1)},\ldots,x^{(p)}}$ is orthogonal and $q=p$, we fall into the setting of Section~\ref{casortho}. An alternative in this case is to use $\hmu$ given by~\eref{fasteq}, which is easy to compute numerically. 

\subsection{Estimation of BV functions}\label{exhaar}
We consider here the functional setting  
\begin{equation} \label{funct}
\mu_{i}=f(x_{i}),\quad i=1,\ldots ,n
\end{equation}
 where $f:[0,1]\to\R$ is an unknown function and $x_{1},\ldots,x_{n}$ are  $n$ deterministic points of $[0,1]$. We assume for simplicity  that $0=x_{1}<x_{2}<\cdots<x_{n}<x_{n+1}=1$ and $n=2^{J_{n}}\geq 8$.  We set $J^*=J_{n}-1$ and $\Lambda^*=\cup_{j=0}^{J^*} \Lambda(j)$ with $\Lambda(0)=\ac{(0,0)}$ and $\Lambda(j)=\ac{j}\times\ac{0,\ldots, 2^{j-1}-1}$ for $j\geq 1$. For $(j,k)\in\Lambda^*$ we define $v_{j,k}\in\R^n$ by
$$\cro{v_{j,k}}_{i}=2^{(j-1)/2}\pa{{\bf 1}_{I^+_{j,k}}(i)-{\bf 1}_{I^-_{j,k}}(i)},\ \ i=1,\ldots,n $$
with
$I^+_{j,k}=\ac{1+(2k+1)2^{-j}n,\ldots,(2k+2)2^{-j}n}$ and 
$I^-_{j,k}=\ac{1+2k2^{-j}n,\ldots,(2k+1)2^{-j}n}$. The family $\ac{v_{j,k},\ (j,k)\in\Lambda^*}$ corresponds to the image of the points $x_{1},\ldots,x_{n}$
by a Haar basis (see Section \ref{proofhaar}) and it is orthonormal 
 for the scalar product 
$$<x,y>_{n}={1\over n}\sum_{i=1}^nx_{i}y_{i}.$$
 We use the collection of models $\S_{m}={\rm span}\ac{v_{j,k},\ (j,k)\in m}$ indexed by $\M=\mathcal P(\Lambda^*)$ and  fall into the setting of Section \ref{casortho}.  We choose the distribution $\pi$ given by~\eref{orthopoids} with $p=n/2$ and $\alpha=1$. We also set $b=1$
 and take some $\beta$ fulfilling
$$\beta\leq{1\over 2}\phi^{-1}\pa{2\log(n/2)\over n}.$$
According to Proposition \ref{fast} the estimator \eref{estimateur} then takes the form
\begin{equation}\label{easy}
\hmu=\sum_{j=0}^{J^*}\sum_{k\in\Lambda(j)}\pa{{Z_{j,k}\,\exp\pa{n\beta Z_{j,k}^2/{\hat \sigma^2}}\over en/2+\exp\pa{n\beta Z_{j,k}^2/{\hat \sigma^2}}}}v_{j,k}\ ,
\end{equation}
with $Z_{j,k}=<Y,v_{j,k}>_{n}$ and
$$\hat\sigma^2=2\pa{<Y,Y>_{n}^2-\sum_{j=0}^{J^*}\sum_{k\in\Lambda(j)}Z_{j,k}^2}.$$

Next corollary gives the rate of convergence of this estimator  when $f$ has bounded variation, in terms of the norm $\|\cdot\|_{n}$ induced by the scalar product $<\cdot,\cdot >_{n}$.
\begin{corollary}\label{corhaar}
In the setting described above, there exists a numerical constant $C$ such that for any function  $f$ with bounded variation $V(f)$
$$\E\pa{\|\mu-\hmu\|^2_{n}}\leq C\max\ac{\pa{V(f)\sigma^2\log n \over n}^{2/3}\ ,\ {V(f)^2\over n}\ ,\ {\sigma^2\log n \over n}}.$$
\end{corollary}
The proof is delayed to Section \ref{proofhaar}. The minimax rate in this setting is $(V(f)\sigma^2/n)^{2/3}$. So,  the rate of convergence of the estimator differs from the minimax rate by a $(\log n)^{2/3}$ factor. We can actually obtain a rate-minimax estimator by using a smaller collection of models similar to the one  introduced in the next section, but we lose then Formula~\eref{easy}.

\subsection{Regression on Besov space $\mathcal{B}^\alpha_{p,\infty}[0,1]$}\label{besov}
We consider again the setting \eref{funct} with $f:[0,1]\to\R$ and introduce a $L^2([0,1],dx)$-orthonormal family 
$\ac{\phi_{j,k},j\geq 0, k=1\ldots2^j}$ of compactly support wavelets with regularity $r$.
We will use models generated by finite subsets of wavelets. If we want that our estimator shares some good adaptive properties on Besov spaces, we shall introduce a family of models induced by the compression algorithm of Birg\'e and Massart \cite{BM01}. This collection turns to be slightly more intricate than the family used in the previous section. We start with some $\kappa<1$ and set
$J_{*}=\lfloor{\log(\kappa n/2)/ \log2 }\rfloor$. 
The largest approximation space we will consider is
$$\F_{*}=\textrm{span}\ac{\phi_{j,k},j=0\ldots J_*,k=1\ldots 2^j},$$
whose dimension  is bounded by $\kappa n$.
For $1\leq J\leq J_{*}$, we define
$$\M_{J}=\ac{m=\bigcup_{j=0}^{J_{*}}\ac{j}\times A_{j}, \textrm{ with } 
A_{j}\in\Lambda_{j,J}},$$
where $\Lambda_{j,J}=\ac{\ac{1,\ldots,2^j}}$ when $j\leq J-1$ and  
$$\Lambda_{j,J}=\ac{A\subset \ac{1,\ldots,2^j}:\ |A|=\lfloor{2^J/(j-J+1)^3\rfloor}},\quad
\textrm{when}\ J\leq j\leq J_{*}.$$
To $m$ in $\M=\cup_{J=1}^{J_{*}}\M_{J}$, we associate
$\F_{m}=\textrm{span}\ac{\phi_{j,k},\; (j,k)\in m}$
and define the model $\S_{m}$ by
$$\S_{m}=\ac{(f(x_{1}),\ldots,f(x_{n})),\; f\in\F_{m}}\subset \S_{*}=\ac{(f(x_{1}),\ldots,f(x_{n})),\; f\in\F_{*}}.$$
When $m\in \M_{J}$, the dimension  of $\S_{m}$ is bounded from above by
\begin{equation}\label{dimension}
\dim(\S_{m})\leq \sum_{j=0}^{J-1}2^j+\sum_{j=J}^{J_{*}}{2^J\over (j-J+1)^3}\leq 2^J\cro{1+\sum_{k=1}^{J_{*}-J+1}k^{-3}}\leq2.2\cdot 2^J
\end{equation}
and $\dim(\S_{*})\leq \kappa n$.
Note also that  the cardinality of $\M_{J}$ is
$$|M_{J}|=\prod_{j=J}^{J_{*}}\binom{2^j}{\lfloor 2^J/(j-J+1)^3\rfloor}.$$
To estimate $\mu$, we use the estimator $\hmu$ given by \eref{estimateur} with $\beta$ given by \eref{betaopt} and
$$L_{m}=1.1\cdot 2^J\quad\textrm{and} \quad\pi_{m}=\cro{2^J(1-2^{J_{*}})\prod_{j=J}^{J_{*}}\binom{2^j}{\lfloor 2^J/(j-J+1)^3\rfloor}}^{-1},\quad\textrm{for}\ m\in\M_{J}.$$

Next corollary gives the rate of convergence of the estimator $\hmu$
 when $f$ belongs to some Besov ball $\mathcal{B}^\alpha_{p,\infty}(R)$  with $1/p<\alpha<r$ (we refer to De Vore and Lorentz \cite{DVL93} for a precise definition of Besov spaces). As it is usual in this setting, we express the result in terms of the norm $\|\cdot\|^2_{n}=\|\cdot\|^2/n$ on $\R^n$.
\begin{corollary}\label{corbesov}
For any $p,R>0$ and $\alpha\in]1/p,r[$,
 there exists some constant $C$
 not depending on $n$ and $\sigma^2$
  such that the estimator $\hmu$ defined above
fulfills
$$\E\pa{\|\mu-\hat \mu\|^2_{n}}\leq C \max\ac{\pa{\sigma^2\over n}^{2\alpha/(2\alpha+1)}\ ,\ {1\over n^{2(\alpha-1/p)}}\ ,\ {\sigma^2\over n}}$$
for any $\mu$  given by \eref{funct} with  $f\in\mathcal{B}^\alpha_{p,\infty}(R)$.
\end{corollary}
The proof is delayed to Section \ref{proofbesov}.
We remind that the rate $\pa{\sigma^2/ n}^{2\alpha/(2\alpha+1)}$ is minimax in this framework, see Yang and Barron \cite{YB99}.

\section{A numerical illustration}\label{illustration}
We illustrate the use of our procedure on a numerical simulation. We start from a signal
$$f(x)=0.7\cos(x)+\cos(7x)+1.5\sin(x)+0.8\sin(5x)+0.9\sin(8x), \quad x\in[0,1]$$
which is in black in Figure~\ref{illustre}. We have $n=60$ noisy observations of this signal 
$$Y_{i}=f(x_{i})+\sigma\eps_{i},\quad i=1,\ldots,60,\qquad\textrm{(the crosses  in Figure~\ref{illustre})}$$
where $x_{i}=i/60$ and $\eps_{1},\ldots,\eps_{60}$ are 60 i.i.d.\;standard Gaussian random variables. The noise level $\sigma$ is not known ($\sigma$ equals 1 in Figure~\ref{illustre}).
\begin{figure}
     \includegraphics[width=11cm]{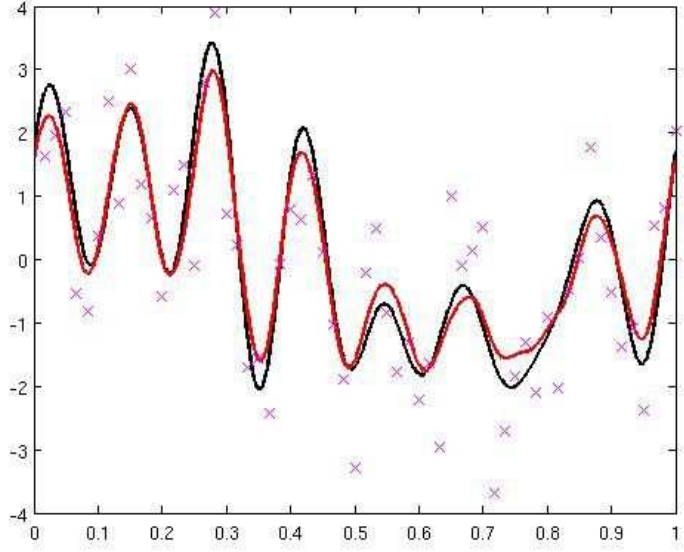}
   \caption{\label{illustre}Recovering a signal from noisy observations (the crosses). The signal is in black, the estimator is in red.}
\end{figure}
To estimate $f$ we will expand the observations on the Fourier basis 
$$\ac{1,\cos(2\pi x),\ldots,\cos(40\pi x),\sin(2\pi x),\ldots,\sin(40\pi x)}.$$
In this direction, we introduce the $p=41$ vectors $v_{1},\ldots,v_{41}$ given by
$$v_{j}=\left\{\begin{array}{ll}
\pa{\sqrt{2\over n}\,\sin(2\pi j x_{1}),\ldots,\sqrt{2\over n}\,\sin(2\pi j x_{n})}'&\textrm{when }j\in\ac{1,\ldots,20}\\
\pa{\sqrt{1\over n},\ldots,\sqrt{1\over n}}'&\textrm{when }j=21\\
\pa{\sqrt{2\over n}\,\cos(2\pi (j-21) x_{1}),\ldots,\sqrt{2\over n}\,\cos(2\pi (j-21) x_{n})}'&\textrm{when }j\in\ac{22,\ldots,41}.
\end{array}\right.$$
This vectors $\ac{v_{1},\ldots,v_{41}}$ form an orthonormal family for the usual scalar product of $\R^n$ and we fall into the setting of Section~\ref{casortho}.

We estimate $(f(x_{1}),\ldots,f(x_{n}))'$ with $\hmu$ given by~\eref{fasteq} with the parameter $\alpha=1$, $b=1$ and $\beta=1/3$. Finally, we estimate $f$ with 
$$\hat f(x)= \hat a_{0}+\sum_{j=1}^{20}\hat a_{j}\cos(2\pi j x)+\sum_{j=1}^{20}\hat b_{j}\sin(2\pi j x)\qquad\textrm{(in red in Figure~\ref{illustre})}$$
where $\hat a_{0}=\sqrt{(1/ n)}<\hmu,v_{21}>$ and 
  $\hat a_{j}=\sqrt{(2/ n)}<\hmu,v_{j+21}>$, $\ \hat b_{j}=\sqrt{(2/ n)}<\hmu,v_{j}>$ for $j=1,\ldots,20$. The plot of $\hat f$ is in red in Figure~\ref{illustre}.

\section{Proofs}
\subsection{Proof of Proposition \ref{fast}}\label{preuvefast}
Let us first express the weights $w_{m}$ in terms of the $Z_{j}$s and $\tau=\alpha\log p+b$. We use below the convention, that the sum of zero term is equal to zero. Then for any $m\in\M$, we have
\begin{eqnarray*}
w_{m}&=&{\exp\pa{\beta \|\hmu_{m}\|^2/\hat \sigma^2-(\alpha\log p+b)|m|}\over \sum_{m'\in\M}\exp\pa{\beta \|\hmu_{m'}\|^2/\hat \sigma^2-(\alpha\log p+b)|m'|}} \\
&=& {\exp\pa{\sum_{k\in m}\pa{{\beta}Z_{k}^2/ \hat \sigma^2-\tau}}\over  \sum_{m'\in\M} \exp\pa{\sum_{k\in m'}\pa{\beta Z_{k}^2/\hat \sigma^2-\tau} }}\ .
\end{eqnarray*}
We write $\M_{j}$ for the set of all the subsets of $\ac{1,\ldots,j-1,j+1,\ldots,p}$. 
Then, for any $j\in\ac{1,\ldots,p}$
\begin{eqnarray*}
c_{j}&=&{\sum_{m\in\M}\1_{j\in m} w_{m} }\\
&=&{\sum_{m\in\M}\1_{j\in m} \exp\pa{\sum_{k\in m}\pa{{\beta}Z_{k}^2/ \hat \sigma^2-\tau}}\over \sum_{m\in\M} \exp\pa{\sum_{k\in m}\pa{\beta Z_{k}^2/\hat \sigma^2-\tau} }}.
\end{eqnarray*}
Note that any subset $m\in\M$  with $j$ inside can be written as
$\ac{j}\cup m'$ with $m'\in \M_{j}$, so that
\begin{eqnarray*}
c_{j}&=&{e^{{\beta}Z_{j}^2/ \hat \sigma^2-\tau}\sum_{m'\in\M_{j}} \exp\pa{\sum_{k\in m'}\pa{{\beta}Z_{k}^2/ \hat \sigma^2-\tau}}\over e^{{\beta}Z_{j}^2/ \hat \sigma^2-\tau}\sum_{m'\in\M_{j}} \exp\pa{\sum_{k\in m'}\pa{{\beta}Z_{k}^2/ \hat \sigma^2-\tau}}  + \sum_{m\in\M_{j}} \exp\pa{\sum_{k\in m}\pa{{\beta}Z_{k}^2/ \hat \sigma^2-\tau}}  }\\
&=& {e^{{\beta}Z_{j}^2/ \hat \sigma^2-\tau}\over {e^{{\beta}Z_{j}^2/ \hat \sigma^2-\tau}}+1}.
\end{eqnarray*}
Formula \eref{fasteq} follows.

\subsection{A preliminary lemma}
\begin{lemma}\label{deviation}
Consider an integer $N$ larger than $2$ and a random variable $X$, such that $NX$ is distributed as a $\chi^2$ of dimension $N$. Then, for any  $0<a<1$, 
\begin{equation}\label{control}
\E\cro{\pa{{a}-X}_{+}}\leq\E\cro{\pa{{a\over X}-1}_{+}}\leq {2\over (1-a)(N-2)}\exp\pa{-{N}\phi(a)}
\end{equation}
with $\phi(a)={1\over 2}\pa{a-1-\log a}>{1\over 4}(1-a)^2$.
\end{lemma}
{\bf Proof:} Remind first that when $g:\R^+\to\R^+$ and $f:\R^+\to\R^+$ have opposite monotonicity,
$$\E[g(X)f(X)]\leq \E[g(X)]\,\E[f(X)].$$
Setting $g(x)=x$ and $f(x)=(a/x-1)_{+}$ leads to the first inequality since $\E(X)=1$.

We now turn to the second inequality. To start with, we note that
\begin{eqnarray*}
\E\cro{\pa{{a\over X}-1}_{+}}&=&a\E\cro{\pa{{1\over X}-{1\over a}}_{+}}\\
&=&a\int_{1/a}^{+\infty}\p\pa{{X}\leq {1\over t}}\,dt\,.
\end{eqnarray*}
Markov inequality gives for any $\lambda\geq 0$
$$\p\pa{{X}\leq {1\over t}}\leq e^{\lambda/ t}\E\pa{e^{-\lambda X}}=e^{\lambda/ t}\pa{1+{2\lambda\over N}}^{-N/2}$$
and choosing $\lambda=N(t-1)/2$ leads to
\begin{equation}\label{fdeviation}
\p\pa{{X}\leq {1\over t}}\leq{1\over t^{N/2}}\,\exp\pa{{N\over 2}(1-1/t)}\leq \exp\pa{-N\phi(1/t)}
\end{equation}
for any $t>1$. Putting pieces together ensures the bound
\begin{eqnarray*}
\E\cro{\pa{{a\over X}-1}_{+}}&\leq &a\,\int_{1/a}^{+\infty}\exp\pa{{N\over 2}(1-1/t)}{dt\over t^{N/2}}\\
&\leq& a\,\int_{0}^{a}\exp\pa{{N\over 2}(1-x)}x^{N/2-2}{dx}
\end{eqnarray*}
for any $0<a<1$. 
Iterating integrations by parts leads to
 \begin{eqnarray*}
 \E\cro{\pa{{a\over X}-1}_{+}}&=&\exp\pa{{N\over 2}(1-a)}{2\,a^{N/2}\over N-2}\sum_{k\geq 0}{a^k(N/2)^k\over \prod_{i=0}^{k-1}(N/2+i)}\\
&\leq&\exp\pa{{N\over 2}(1-a)}{2\,a^{N/2}\over (N-2)(1-a)}\\
&\leq&{2\over (1-a)(N-2)}\exp\pa{-{N}\phi(a)}
\end{eqnarray*}
for $0<a<1$, and the bound \eref{control} follows.

\subsection{Proof of Theorem \ref{resultat}}\label{preuveresultat}
To keep formulas  short, we   write  $d_{m}$ for the dimension of $\S_{m}$, and use the following notations for the various projections
$$\hmu_{*}=\Pi_{\S_{*}}Y,\quad \mu_{*}=\Pi_{\S_{*}}\mu\quad\textrm{and}\quad \mu_{m}=\Pi_{\S_{m}}\mu,\quad m\in\M.$$
It is also convenient to write $w_{m}$ in the form
$$ w_{m}={\pi_{m}\over \Z'}\exp\pa{-\beta{\|\hmu_{*}-\hmu_{m}\|^2\over \hat \sigma^2}-{L_{m}}}$$
with $\Z'=e^{-\beta\|\hmu_{*}\|^2/\hat \sigma^2}\Z$.

By construction, the estimator $\hat \mu$ belongs to $\S_{*}$, so according to Pythagorean equality we have $\|\mu-\hat\mu\|^2=\|\mu-\mu_{*}\|^2+\|\mu_{*}-\hat\mu\|^2$. The first part is non random, and we only need to control the second part. 

According to Theorem 1 in Leung and Barron \cite{LB06}, the Stein's unbiased estimate of the $L^2$-risk $\E\pa{{\|\mu_{*}-\hmu\|^ {2}}}/\sigma^2$
of the estimator $\hat \mu$ on $\S_{*}$
can be written as
$$S\pa{\hat \mu}=\sum_{m\in\M}w_{m}\cro{{\|\hat\mu_{*}-\hmu_{m}\|^{2}\over \sigma^2}+2d_{m}-d_{*}-{\|\hmu_{m}-\hat\mu\|^{2}\over \sigma^2}+2\beta\nabla\pa{{\|\hmu_{*}-\hmu_{m}\|^{2}\over \hat\sigma^2}}.\pa{\hat \mu-\hat\mu_{m}}}.$$
When expanding the gradient into
$$\nabla\pa{{\|\hmu_{*}-\hmu_{m}\|^{2}\over \hat\sigma^2}}.\pa{\hat \mu-\hat\mu_{m}}={2\pa{\hat\mu_{*}-\hat\mu_{m}}\over \hat\sigma^2}.\pa{\hat \mu-\hat\mu_{m}}+\|\hmu_{*}-\hmu_{m}\|^{2}\nabla\pa{{1/ \hat\sigma^2}}.\pa{\hat \mu-\hat\mu_{m}},$$
 the term
$\|\hmu_{*}-\hmu_{m}\|^{2}\nabla\pa{{1/ \hat\sigma^2}}.\pa{\hat \mu-\hat\mu_{m}}$ turns to be 0, since $\nabla\pa{{1/ \hat\sigma^2}}$ is orthogonal to $\S_{*}$. Furthermore, the sum $$\sum_{m\in\M}w_{m}\pa{\hmu-\hmu_{*}}.\pa{\hmu-\hmu_{m}}$$ also equals 0, so an unbiased estimate of  $\E\pa{{\|\mu_{*}-\hmu\|^ {2}}}/\sigma^2$ is
$$S\pa{\hat \mu}=\sum_{m\in\M}w_{m}\cro{{\|\hat\mu_{*}-\hmu_{m}\|^{2}\over \sigma^2}+2d_{m}+\pa{4\beta-{\hat\sigma^2\over\sigma^2}}{\|\hmu_{m}-\hat\mu\|^{2}\over \hat\sigma^2}}-d_{*}.$$
We control the last term thanks to the upper bound
\begin{eqnarray*}
\sum_{m\in\M}w_{m}{\|\hmu_{m}-\hat\mu\|^{2}}&=&\sum_{m\in\M}w_{m}{\|\hmu_{*}-\hat\mu_{m}\|^{2}}-\|\hmu-\hat\mu_{*}\|^{2}\\
&\leq&\sum_{m\in\M}w_{m}{\|\hmu_{*}-\hat\mu_{m}\|^{2}}
\end{eqnarray*}
and get
\begin{eqnarray*}
S\pa{\hat \mu}&\leq&\cro{{\hat\sigma^2\over \sigma^2}+\pa{4\beta-{\hat\sigma^2\over\sigma^2}}_{+}}\sum_{m\in\M}w_{m}\,{\|\hat\mu_{*}-\hmu_{m}\|^{2}\over \hat\sigma^2}+2\sum_{m\in\M}w_{m}d_{m}-d_{*}\\
&\leq &\cro{{\hat\sigma^2\over \sigma^2}+\pa{4\beta-{\hat\sigma^2\over\sigma^2}}_{+}}\sum_{m\in\M}w_{m}\cro{{\|\hat\mu_{*}-\hmu_{m}\|^{2}\over \hat\sigma^2}+{L_{m}\over \beta}}-d_{*}\\
&&+\sum_{m\in\M}w_{m}\pa{2d_{m}-\cro{{\hat\sigma^2\over \sigma^2}+\pa{4\beta-{\hat\sigma^2\over\sigma^2}}_{+}}{L_{m}\over\beta}}
\end{eqnarray*}
where $(x)_+=\max(0,x)$.
First note that when $L_{m}\geq d_{m}/2$ we have
\begin{eqnarray*}
2d_{m}-\cro{{\hat\sigma^2\over \sigma^2}+\pa{4\beta-{\hat\sigma^2\over\sigma^2}}_{+}}{L_{m}\over\beta}&\leq &\cro{2-{\hat \sigma^2\over2\beta\sigma^2}-{1\over 2\beta}\pa{4\beta-{\hat\sigma^2\over\sigma^2}}_{+}} d_{m}\\
&\leq& \min\pa{0,2-{\hat \sigma^2\over2\beta\sigma^2}}d_{m}\ \leq \ 0.
\end{eqnarray*}
Therefore, setting $\hd=\pa{{4\beta\sigma^2/\hat\sigma^2}-1}_{+}$ we get 
$$S(\hmu)\leq {\hat\sigma^2\over \sigma^2}\pa{1+\hd}\sum_{m\in\M}w_{m}\cro{{\|\hat\mu_{*}-\hmu_{m}\|^{2}\over \hat\sigma^2}+{L_{m}\over \beta}}-d_{*}.$$
Let us introduce the Kullback divergence between two probability distributions $\ac{\alpha_{m},\ m\in\M}$ and $\ac{\pi_{m},\ m\in\M}$ on $\M$
$$\D(\alpha|\pi)=\sum_{m\in\M}\alpha_{m}\log {\alpha_{m}\over\pi_{m}} \ \geq \ 0$$
 and the function
$$\EL_{\beta}^{\pi}(\alpha)=\sum_{m\in\M}\alpha_{m}\cro{{\|\hat\mu_{*}-\hmu_{m}\|^{2}\over \hat\sigma^2}+{L_{m}\over \beta}}+{1\over \beta}\D(\alpha|\pi).$$
The latter function is convex on the simplex $S_{\M}^+=\ac{\alpha\in [0,1]^{|\M|}, \ \sum_{m\in \M}\alpha_{m}=1}$ and 
can be interpreted as a free energy function. Therefore, it
is minimal for the Gibbs measure $\ac{w_{m},\ m\in\M}$ and for any $\alpha\in S_{\M}^+$, 
\begin{eqnarray*}
S(\hmu)&\leq&  {\hat\sigma^2\over \sigma^2}\pa{1+\hd}\cro{\sum_{m\in\M}\alpha_{m}\cro{{\|\hat\mu_{*}-\hmu_{m}\|^{2}\over \hat\sigma^2}+{L_{m}\over \beta}}+{1\over \beta}\D(\alpha |\pi)-{1\over \beta}\D(w |\pi)}-d_{*}\\
&\leq& \pa{1+\hd}\cro{\sum_{m\in\M}\alpha_{m}\cro{{\|\hat\mu_{*}-\hmu_{m}\|^{2}\over \sigma^2}+{\hat\sigma^2\over \beta\sigma^2}L_{m}}+{\hat\sigma^2\over \beta\sigma^2}\D(\alpha |\pi)}-d_{*}.
%&\leq&\pa{1+\hd}\cro{\sum_{m\in\M}\alpha_{m}\cro{{\|\hat\mu_{*}-\hmu_{m}\|^{2}\over \sigma^2}-d_{*}+{\hat\sigma^2\over \beta\sigma^2}L_{m}}+{\hat\sigma^2\over \beta\sigma^2}\D(\alpha |\pi)}+\hd d_{*}.
\end{eqnarray*}
We fix a  probability distribution $\alpha\in S_{\M}^+$ and  take the expectation in the last inequality to get 
\begin{multline*}
\E\cro{S(\hmu)}\leq\\  \pa{1+\E(\hd)}\sum_{m\in\M}\alpha_{m}{\E\cro{\|\hat\mu_{*}-\hmu_{m}\|^{2}\over \sigma^2}}
+\E\cro{{\hat\sigma^2\over \beta\sigma^2}(1+\hd)}
\cro{\sum_{m\in\M}\alpha_{m}{L_{m}}+\D(\alpha |\pi)}- d_{*}.
\end{multline*}
Since $\hat\sigma^2/\sigma^2$ is stochastically larger than a random variable $X$ with $\chi^2(N_{*})/N_{*}$ distribution,  Lemma \ref{deviation} ensures that the two expectations
$$\E\cro{{\hat\sigma^2\over \sigma^2}\hd}=\E\cro{\pa{4\beta-{\hat\sigma^2\over\sigma^2}}_{+}}\ \textrm{ and }\ \E\cro{\hd}=\E\cro{\pa{{4\beta\sigma^2\over\hat\sigma^2}-1}_{+}}$$
 are bounded by   
$${2\over (1-4\beta)(N_{*}-2)}\exp\pa{-{N_{*}}\phi(4\beta)},$$
with $\phi(x)=(x-1-\log(x))/2$.
Furthermore, the condition $N_{*}\geq 2+(\log n) /\phi(4\beta)$ enforces 
$$ {2\over (1-4\beta)(N_{*}-2)}\exp\pa{-{N_{*}}\phi(4\beta)}\ \leq\ {2\phi(4\beta)e^{-2\phi(4\beta)}\over (1-4\beta)n\log n}\ \leq\ {1\over 2n\log n}\ =\ \varepsilon_{n}.$$
Putting pieces together, we obtain
\begin{eqnarray*}
\lefteqn{\E\cro{{\|\mu-\hmu\|^ {2}}}\over\sigma^2}\\
&=&{\|\mu-\mu_{*}\|^ {2}\over\sigma^2}+\E[S(\hmu)]\\
&\leq& {\|\mu-\mu_{*}\|^ {2}\over\sigma^2}+(1+\varepsilon_{n})\sum_{m\in\M}\alpha_{m}\E\cro{\|\hmu_{*}-\hmu_{m}\|^{2}\over \sigma^2}\\
 &&+\pa{{\bar\sigma^2\over \beta\sigma^2}+\varepsilon_{n}}\cro{\sum_{m\in\M}\alpha_{m}L_{m}+\D(\alpha |\pi)}-d_{*}\\
&\leq&{\|\mu-\mu_{*}\|^ {2}\over\sigma^2}
+(1+\varepsilon_{n})\sum_{m\in\M}\alpha_{m}\cro{{\|\mu_{*}-\mu_{m}\|^{2}\over \sigma^2}+d_{*}-d_{m}+{\bar\sigma^2\over \beta\sigma^2}(L_{m}+\D(\alpha |\pi))}-d_{*}\\
&\leq&(1+\varepsilon_{n})\cro{\sum_{m\in\M}\alpha_{m}\cro{{\|\mu-\mu_{m}\|^{2}\over \sigma^2}-d_{m}+{\bar\sigma^2\over \beta\sigma^2}L_{m}}+{\bar\sigma^2\over \beta\sigma^2}\D(\alpha |\pi)}+\varepsilon_{n}d_{*}.
\end{eqnarray*}
This inequality holds for any non-random probability distribution $\alpha\in S_{\M}^+$, so it holds in particular for the Gibbs measure
$$\alpha_{m}={\pi_{m}\over \Z_{\beta}} \exp\cro{-{\beta\over \bar\sigma^2}\pa{{\|\mu-\mu_{m}\|^{2}}-d_{m}\sigma^2}-L_{m}},\quad m\in\M $$
where $\Z_{\beta}$ normalizes the sum of the $\alpha_{m}$s to one.
For this choice of $\alpha_{m}$ we obtain
$${\E\cro{{\|\mu-\hmu\|^ {2}}}\over\sigma^2}\leq-{(1+\varepsilon_{n})\bar\sigma^2\over \beta\sigma^2}\log\cro{\sum_{m\in\M}\pi_{m}\exp\cro{-{\beta\over \bar\sigma^2}\pa{{\|\mu-\mu_{m}\|^{2}}-d_{m}\sigma^2}-L_{m}}}+\varepsilon_{n}d_{*}$$
which ensures \eref{fgibbs} since $d_{*}\leq n$. To get \eref{foracle} simply note that 
\begin{multline*}
{\sum_{m\in\M}\pi_{m}\exp\cro{-{\beta\over \bar\sigma^2}\pa{{\|\mu-\mu_{m}\|^{2}}-d_{m}\sigma^2}-L_{m}}}\\
\geq {\exp\cro{-{\beta\over \bar\sigma^2}\pa{{\|\mu-\mu_{\m}\|^{2}}-d_{\m}\sigma^2}-L_{\m}-\log \pi_{\m}}}
\end{multline*}
for any $\m\in\M$.

\subsection{Proof of Proposition \ref{resultatvarianceconnue}}\label{preuveprop2}
We use along the proof the notations $\mu_{*}=\Pi_{\S_{*}}\mu$, $\cro{x}_{+}=\max(x,0)$ and $\gamma=\gamma_{\beta}(p)=\sqrt{2+\beta^{-1}\log p}$. 
We omit the proof of the bound $c_{\beta}(p)\leq 16$ for $\beta\in[1/4,1/2]$ and $p\geq 3$. This proof (of minor interest) can be found in the Appendix. 

The risk of $\hmu$ is given by
$$\E\pa{\|\mu-\hmu\|^2}=\|\mu-\mu_{*}\|^2+\sum_{j=1}^p\E\pa{\pa{<\mu,v_{j}>-s_{\beta}(Z_{j}/\sigma)Z_{j}}^2}.$$
Note that $Z_{j}=<\mu,v_{j}>+\sigma <\eps,v_{j}>$, with $<\eps,v_{j}>$ distributed as a standard Gaussian random variable. As a consequence, when $|<\mu,v_{j}>|\leq4\gamma\sigma$ we have 
\begin{equation}\label{queue}
\E\pa{\pa{<\mu,v_{j}>-s_{\beta}(Z_{j}/\sigma)Z_{j}}^2}\leq c_{\beta}(p)\cro{\min(<\mu,v_{j}>^2,\gamma^2\sigma^2)+\gamma^2\sigma^2/p}.
\end{equation}
If we prove the same inequality for $|<\mu,v_{j}>|\geq4\gamma\sigma$, then
\begin{eqnarray*}
\E\pa{\|\mu-\hmu\|^2}&\leq&\|\mu-\mu_{*}\|^2+c_{\beta}(p)\sum_{j=1}^p\min(<\mu,v_{j}>^2,\gamma^2\sigma^2)+\gamma^2\sigma^2\\
&\leq&\|\mu-\mu_{*}\|^2+c_{\beta}(p)\inf_{m\in\M}\cro{\|\mu_{*}-\mu_{m}\|^2+\gamma^2(|m|+1)\sigma^2}.
\end{eqnarray*}
This last inequality is exactly the Bound~\eref{formulevarianceconnue}.

To conclude the proof of Proposition~\ref{resultatvarianceconnue}, we need to check that
$$\E\pa{\pa{x-s_{\beta}(x+Z)(x+Z)}^2}\leq c_{\beta}(p)\gamma^2\qquad\textrm{for }x \geq 4\gamma,$$
where $Z$ is distributed as a standard Gaussian random variable.
We first note that 
\begin{eqnarray*}
\E\pa{\pa{x-s_{\beta}(x+Z)(x+Z)}^2}&\leq& 2x^2\E\pa{\pa{1-s_{\beta}(x+Z)}^2}+2\E(s_{\beta}(x+Z)^2Z^2)\\
 &\leq& 2x^2\E\pa{\pa{1-s_{\beta}(x+Z)}^2}+2.
\end{eqnarray*}
We can bound $\pa{1-s_{\beta}(x+Z)}^2$ as follows
\begin{eqnarray*}
\pa{1-s_{\beta}(x+Z)}^2&=&\pa{{1\over 1+\exp\pa{\beta \cro{(x+Z)^2-\gamma^2}}}}^2\\
&\leq& \exp\pa{-2\beta \cro{(x+Z)^2-\gamma^2}_{+}}\\
&\leq& \exp\pa{-{1\over 2} \cro{(x+Z)^2-\gamma^2}_{+}}
\end{eqnarray*}
where the last inequality comes from $\beta\geq 1/4$. We then obtain
\begin{eqnarray*}
\E\pa{\pa{x-s_{\beta}(x+Z)(x+Z)}^2}&\leq&2+4x^2\E\pa{\exp\pa{-{1\over 2} \cro{(x-Z)^2-\gamma^2}_{+}}{\bf 1}_{Z>0}}\\
&\leq&2+4x^2\cro{\p\pa{0<Z<x/2}\exp\pa{-{1\over 2} \cro{(x/2)^2-\gamma^2}}+\p\pa{Z>x/2}}\\
&\leq& 2+2x^2\exp\pa{-x^2/8+\gamma^2/2}+4x^2\exp\pa{-x^2/8}.
\end{eqnarray*}
When $p\geq 3$ and $\beta\leq 1/2$, we have $\gamma\geq \sqrt{2+2\log 3}$ and then
$$\sup_{x>4\gamma}x^2e^{-x^2/8}=16\gamma^2\exp(-2\gamma^2).$$
Therefore, 
$$\E\pa{\pa{x-s_{\beta}(x+Z)(x+Z)}^2}\leq 2+16\gamma^2\pa{2e^{-3\gamma^2/2}+4e^{-2\gamma^2}}\leq 0.6\,\gamma^2\leq c_{\beta}(p)\gamma^2,$$
where we used again the bound $\gamma^2\geq 2+2\log 3$.
The proof of Proposition~\ref{resultatvarianceconnue} is  complete.

\subsection{Proof of Corollary \ref{corhaar}}\label{proofhaar}
We start by proving some results on the approximation of BV functions  with the Haar wavelets.
\subsubsection{Approximation of BV functions}
We stick to the setting of Section \ref{exhaar} with  $0=x_{1}<x_{2}<\cdots<x_{n}<x_{n+1}=1$ and $n=2^{J_{n}}$. For $0\leq j\leq J_{n}$ and $p\in\Lambda(j)$ we define $t_{j,p}=x_{p2^{-j}n+1}$.
We also set $\phi_{0,0}=1$ and
$$\phi_{j,k}=2^{(j-1)/2}\pa{{\bf 1}_{[t_{j,2k+1},t_{j,2k+2})}-{\bf 1}_{[t_{j,2k},t_{j,2k+1})}},\ \textrm{for }1\leq j\leq J_{n}\ \textrm{and}\ k\in\Lambda(j).$$
This family of Haar wavelets is orthonormal for the positive semi-definite quadratic form
$$(f,g)_{n}={1\over n}\sum_{i=1}^nf(x_{i})g(x_{i})$$
on functions mapping $[0,1]$ into $\R$.
For $0\leq J\leq J_{n}$, we write $f_{J}$ for the projection of $f$ onto the linear space spanned by $\ac{\phi_{j,k}, \ 0\leq j\leq J,\ k\in\Lambda(j)}$ with respect to $(\cdot,\cdot)_{n}$, namely
$$f_{J}=\sum_{j=0}^J \sum_{k\in\Lambda(j)}c_{j,k}\phi_{j,k},\quad
\textrm{with}\ c_{j,k}=(f,\phi_{j,k})_{n}.$$
We also consider for $1\leq J\leq J_{n}$ an approximation of $f$ {\it \`a la} Birg\'e and Massart \cite{BM00}
$$\tilde f_{J}=f_{J-1}+\sum_{j=J}^{J_{n}}\sum_{k\in\Lambda'_{J}(j)}c_{j,k}\phi_{j,k},$$
where $\Lambda'_{J}(j)\subset\Lambda(j)$ is the set of indices $k$ we obtain when we select the $K_{j,J}$ largest coefficients $|c_{j,k}|$ amongs  $\ac{|c_{j,k}|,\ k\in\Lambda(j)}$, with $K_{j,J}=\lfloor(j-J+1)^{-3}2^{J-2}\rfloor$ for  $1\leq J\leq j\leq J_{n}$. Note that the number of coefficients $c_{j,k}$ in $\tilde f_{J}$ is bounded from above by
 $$1+\sum_{j=1}^{J-1}2^{j-1}+\sum_{j\geq J} (j-J+1)^{-3}2^{J-2}\leq 2^{J-1}+2^{J-2}\sum_{p\geq 1}p^{-3}\leq 2^J.$$
 Next proposition states approximation bounds for $f_{J}$ and $\tilde f_{J}$ in terms of  the (semi-)norm $\|f\|^2_{n}=(f,f)_{n}$.
\begin{proposition}\label{approxprop}
When $f$ has bounded variation $V(f)$, we have
\begin{equation}\label{approxlineaire}
\|f-f_{J}\|_{n}\leq 2V(f) 2^{-J/2},\quad{\rm for}\ J\geq 0
\end{equation}
and
\begin{equation}\label{approxBM}
\|f-\tilde f_{J}\|_{n}\leq c\,V(f)2^{-J},\quad \textrm{for}\ J\geq 1
\end{equation}
with $c=\sum_{p\geq 1}p^32^{-p/2+1}$.
\end{proposition}
Formulaes \eref{approxlineaire} and \eref{approxBM} are based on the following fact.
\begin{lemma}\label{clef}
When $f$ has bounded variation $V(f)$, we have
$$\sum_{k\in\Lambda(j)}|c_{j,k}|\leq 2^{-(j+1)/2}V(f),\ \quad {\rm for} \ 1\leq j\leq J_{n}.$$
\end{lemma}
{\it Proof of the Lemma.} We assume for simplicity that $f$ is non-decreasing. Then, we have
$$c_{j,k}=<f,\phi_{j,k}>_{n}= {2^{(j-1)/2}\over n}\cro{\sum_{i\in I^+_{j,k}}f(x_{i})-\sum_{i\in I^-_{j,k}}f(x_{i})},$$
with $I^+_{j,k}$ and $I^-_{j,k}$ defined in Section \ref{exhaar}. Since 
$|I^+_{j,k}|=2^{-j}n$ and $f$ is non-decreasing
\begin{eqnarray*}
|c_{j,k}|&\leq& {2^{(j-1)/2}\over n} |I^+_{j,k}|\cro{f(x_{(2k+2)2^{-j}n})-f(x_{(2k)2^{-j}n})}\\
&\leq& 2^{-(j+1)/2}\cro{f(x_{(2k+2)2^{-j}n})-f(x_{(2k)2^{-j}n})},
\end{eqnarray*}
and Lemma \ref{clef} follows. \hfill $\square$

We first prove \eref{approxlineaire}. Since the $\ac{\phi_{j,k},\ k\in\lambda(j)}$ have disjoint supports, we have for $0\leq J\leq J_{n}$
\begin{eqnarray*}
\|f-f_{J}\|_{n}&\leq& \sum_{j>J}\|\sum_{k\in\Lambda(j)}c_{j,k}\phi_{j,k}\|_{n}\\
&\leq& \sum_{j>J}\cro{\sum_{k\in\Lambda(j)}|c_{j,k}|^2\underbrace{\|\phi_{j,k}\|_{n}^2}_{=1}}^{1/2}\\
&\leq&  \sum_{j>J}\sum_{k\in\Lambda(j)}|c_{j,k}|.
\end{eqnarray*}
Formula \eref{approxlineaire} then follows from Lemma \ref{clef}.

To prove \eref{approxBM} we introduce the set $\Lambda_{J}''(j)=\Lambda(j)\setminus\Lambda_{J}'(j)$. Then, for $1\leq J\leq J_{n}$ we have
\begin{eqnarray*}
\|f-\tilde f_{J}\|_{n}&\leq&  \sum_{j=J}^{J_{n}}\cro{\sum_{k\in\Lambda_{J}''(j)}|c_{j,k}|^2\underbrace{\|\phi_{j,k}\|_{n}^2}_{=1}}^{1/2}\\
&\leq&  \sum_{j=J}^{J_{n}}\cro{\max_{k\in\Lambda_{J}''(j)}|c_{j,k}|\sum_{k\in\Lambda(j)}|c_{j,k}|}^{1/2}.
\end{eqnarray*}
The choice of $\Lambda_{J}'(j)$ enforces the inequalities
\begin{eqnarray*}
(1+K_{j,J})\max_{k\in\Lambda_{J}''(j)}|c_{j,k}|\leq\sum_{k\in\Lambda_{J}''(j)}|c_{j,k}|+\sum_{k\in\Lambda_{J}'(j)}|c_{j,k}|\leq\sum_{k\in\Lambda(j)}|c_{j,k}|.
\end{eqnarray*}
To complete the proof of Proposition \ref{approxprop}, we combine this bound with Lemma \ref{clef}:
\begin{eqnarray*}
\|f-\tilde f_{J}\|_{n}&\leq&\sum_{j\geq J}2^{-(j+1)/2}V(f)(1+K_{j,J})^{-1/2}\\
&\leq& \sum_{j\geq J}2^{-(j+1)/2}V(f)2^{-(J-2)/2}(j-J+1)^3.\\
&\leq& V(f)2^{-J}{\sum_{p\geq 1}p^32^{-p/2+1}}.
\end{eqnarray*}

\subsubsection{Proof of Corollary \ref{corhaar}}
First, note that $v_{j,k}=(\phi_{j,k}(x_{1}),\ldots,\phi_{j,k}(x_{n}))'$ for $(j,k)\in\Lambda^*$. Then,
according to \eref{approxlineaire} and \eref{approxBM} there exists  for any $0\leq J\leq J^*$  a model $m\in\M$ fufilling  $|m|\leq 2^J$ and 
\begin{eqnarray*}
\|\mu-\Pi_{\S_{m}}\mu\|_{n}^2&=&\|\mu-\Pi_{\S_{*}}\mu\|^2_{n}+\|\Pi_{\S_{*}}\mu-\Pi_{\S_{m}}\mu\|^2_{n}\\
&\leq& 2c^2 V(f)^2 \pa{2^{-J^*}\vee 2^{-2J}},
\end{eqnarray*}
with $c=\sum_{p\geq 1}p^32^{-p/2+1}$.
Putting together this approximation result with Theorem~\ref{resultat} gives
$$\E\pa{\|\mu-\hmu\|^2_{n}}\leq C \inf_{0\leq J\leq J^*} \cro{V(f)^2 \pa{2^{-J^*}\vee 2^{-2J}}+{2^J\log n\over n}\ \sigma^2}$$
for some numerical constant $C$, when 
$$\beta\leq {1\over 4}\phi^{-1}\pa{\log n\over n/2-2}.$$ This bound still holds true when
$$ {1\over 4}\phi^{-1}\pa{\log n\over n/2-2}\leq \beta\leq {1\over 2}\phi^{-1}\pa{\log (n/2)\over n/2},$$ 
see Proposition \ref{resultatortho} in the Appendix.
To conclude the proof of Corollary \ref{corhaar}, we apply the previous bound with $J$ given by the minimum between $J^*$ and the smallest integer such that
$$2^J\geq V(f)^{2/3}\pa{n\over \sigma^2\log n}^{1/3}.$$

\subsection{Proof of Corollary \ref{corbesov}}\label{proofbesov}
First, according to the inequality $\binom{n}{k}\leq(en/k)^k$ we have the bound for $m\in\M_{J}$
\begin{eqnarray}
-\log \pi_{m}&\leq& \log 2^J+\sum_{j=J}^{J_{*}}{2^J\over (j-J+1)^3}\ \log\pa{e2^{j-J+1}(j-J+1)^3}\nonumber\\
&\leq& 2^J\pa{1+\sum_{k\geq 1}k^{-3}\pa{1+3\log k+k\log 2}}\nonumber\\
&\leq& 4\cdot 2^J.\label{pbesov}
\end{eqnarray}

Second, when $f$ belongs to some Besov ball $\mathcal{B}^\alpha_{p,\infty}(R)$  with $1/p<\alpha<r$,  Birg\'e and Massart \cite{BM00} gives the following approximation results. There exists a constant $C>0$, such that for any $J\leq J_{*}$ and $f\in \mathcal{B}^\alpha_{p,\infty}(R)$, there exists $m\in\M_{J}$ fulfilling 
$$\|f-\bar \Pi_{\F_{m}}f\|_{\infty}\leq \|f-\bar \Pi_{\F_{*}}f\|_{\infty}+\|\bar \Pi_{\F_{*}}f-\bar \Pi_{\F_{m}}f\|_{\infty}\leq C\max\pa{2^{-J_{*}(\alpha-1/p)},2^{-\alpha J}}$$
where $\bar\Pi_{\F}$ denotes the orthogonal projector onto $\F$ in $L^2([0,1])$. In particular, under the previous assumptions we have
\begin{eqnarray}\label{biais}
\|\mu-\Pi_{\S_{m}}\mu\|^2_{n}&\leq& {1\over n}\sum_{i=1}^n\cro{f(x_{i})-\bar \Pi_{\F_{m}}f(x_{i})}^2\nonumber \\
&\leq& \|f-\bar \Pi_{\F_{m}}f\|^2_{\infty}\leq C^2  \max\pa{2^{-2\alpha J},2^{-2J_{*}(\alpha-1/p)}},
\end{eqnarray}
To conclude the proof of Corollary \ref{corbesov}, we combine Theorem \ref{resultat} together with \eref{dimension}, \eref{pbesov} and \eref{biais} for
$$J=\min\pa{J_{*}\ ,\ \Big\lfloor {\log [\max(n/\sigma^2,1)]\over (2\alpha+1)\log 2} \Big\rfloor+1}.$$

\newpage

\appendix
\

\vspace{0.5cm}

\centerline{\Large APPENDIX}\vspace{0.5cm}

The Appendix is devoted to the proof of the bound $c_{\beta}(p)\leq 16$ and gives  further explanations on the Remark~4, Section~\ref{parametresortho}. The results are stated in Appendix~\ref{perfortho} and the proofs (of minor interest) can be found in Appendix~\ref{preuves}.

\section{Two bounds}\label{perfortho}
\subsection{When the variance is known}
In this section, we assume that the noise level $\sigma$ is known. 
We remind that $3\leq p\leq n$ and $\ac{v_{1},\ldots,v_{p}}$ is an orthonormal family of vectors in $\R^n$. 
Next lemma gives an upper bound on the Euclidean risk of the estimator 
$$\hmu=\sum_{j=1}^p \pa{{e^{\beta Z_{j}^2/\sigma^2}\over pe^{\beta \lambda}+e^{\beta Z_{j}^2/\sigma^2}}\,Z_{j}}v_{j},\qquad\textrm{with}\ \lambda\geq 2\ \textrm{and}\ Z_{j}=<Y,v_{j}>,\; j=1,\ldots,p.$$
\begin{lemma}\label{resultatvarianceconnue2}
For any $\lambda\geq 2$ and $\beta>0$, we set 
$\gamma^2=\lambda+{\beta^{-1}}\,\log p$. 
\begin{enumerate}
\item[{\bf 1.}] For $1/4\leq \beta\leq 1/2$ and  $a\in\R$ we have the bound
\begin{equation}\label{univariate}
\E\cro{\pa{a-{\exp\pa{\beta (a+\eps)^2}\over p\exp(\beta\lambda)+\exp\pa{\beta (a+\eps)^2}}(a+\eps)}^{2}}\leq 16\cro{\min\pa{a^2,\gamma^2}+\gamma^2/p},
\end{equation}
where $\eps$ is distributed as a standard Gaussian random variable.
\item[{\bf 2.}] As a consequence, for any $0< \beta\leq 1/2$ and $\lambda\geq 2$ the risk of  the estimator $\hmu$  is 
 upper bounded by
\begin{equation}\label{formulevarianceconnue2}
\E\pa{\|\mu-\hmu\|^2}\leq 16\inf_{m\in\M}\cro{{\|\mu-\mu_{m}\|^2}+\gamma^2|m|\sigma^2+\gamma^2\sigma^2}.
\end{equation}
\end{enumerate}
\end{lemma}
Note that~\eref{univariate} enforces the bound $c_{\beta}(p)\leq 16$. This constant 16 is certainly far from optimal. Indeed, when $\beta=1/2$, $\lambda=2$ and $3\leq p\leq 10^6$, Figure~1 in Section~\ref{parametresortho} shows that the bounds \eref{univariate} and  \eref{formulevarianceconnue2} hold  with 16 replaced by 1.

\subsection{When the variance is unknown}\label{parametresortho2}
We consider the same framework, except that  the noise level $\sigma$ is not  known.  
Next proposition provides a risk bound for the estimator 
\begin{equation}\label{fasteq2}
\hmu=\sum_{j=1}^{p}(c_{j}Z_{j})v_{j},\quad \textrm{with}\quad  Z_{j}=<Y,v_{j}>\textrm{and}\quad c_{j}={\exp\pa{\beta Z_{j}^2/{\hat \sigma^2}}\over p\exp\pa{b}+\exp\pa{\beta Z_{j}^2/{\hat \sigma^2}}}\,.
\end{equation}
We remind that 
%\begin{equation}\label{phi}
$\phi (x)=(x-1-\log x)/2\,.$
%\end{equation}

\begin{proposition}\label{resultatortho}
Assume  that $\beta$ and $p$ fulfill the conditions, 
\begin{equation}\label{conditionsortho}
p\geq 3,\quad 0<\beta< 1/2 \quad \textrm{and}\quad p+{\log p\over \phi(2\beta)}\leq n.
\end{equation}
Assume also that $b$ is not smaller than 1. Then, we have  the following upper bound on the $L^2$-risk of the estimator $\hmu$ defined by \eref{fasteq2}\begin{equation}\label{formulevarianceinconnue}
\E\pa{\|\mu-\hmu\|^2}
\leq
16\inf_{m\in\M}\cro{{\|\mu-\mu_{m}\|^2}+{1\over \beta}\pa{b+\log p}(|m|+1)\bar\sigma^2+(2+b+\log p)\sigma^2},
\end{equation}
where $\bar\sigma^2=\sigma^2+\|\mu-\Pi_{\S_{*}}\mu\|^2/(n-p)$.
\end{proposition}
We postpone the proof of Proposition \ref{resultatortho} to Section \ref{preuveT2}.

\section{Proofs}\label{preuves}
\subsection{Proof of Lemma \ref{resultatvarianceconnue2}}\label{preuveprop2full}
When $\beta\leq 1/4$ the bound~\eref{formulevarianceconnue2}  follows from a slight variation of Theorem 5 in \cite{LB06}. When  $1/4<\beta\leq 1/2$, it follows from~\eref{univariate} (after a  rescaling by $\sigma^2$ and a summation). 
So all we need is to prove \eref{univariate}. For symmetry reason, we restrict to the case $a>0$.

To prove Inequality \eref{univariate}, we first note that
\begin{eqnarray}
\lefteqn{\E\cro{\pa{a-{\exp\pa{\beta (a+\eps)^2}\over p\exp(\beta\lambda)+\exp\pa{\beta (a+\eps)^2}}(a+\eps)}^2}}\nonumber\\
&=& \E\cro{\pa{{a\over 1+p^{-1}\exp\pa{\beta \cro{(a+\eps)^2-\lambda}}}  - {\eps\over 1+ p\exp(\beta\cro{\lambda-(a+\eps)^2})}}^2}\nonumber\\
&\leq& 2a^2\,\E\cro{\pa{{1\over 1+p^{-1}\exp\pa{\beta \cro{(a+\eps)^2-\lambda}}}}^2}+2\E\cro{\pa{{\eps\over 1+ p\exp(\beta\cro{\lambda-(a+\eps)^2})}}^2}.\label{ineq1}
\end{eqnarray}
and then investigate apart the four cases $0\leq a\leq 2\gamma^{-1}$, $\,2\gamma^{-1}\leq a\leq 1$, $\,1\leq a\leq \sqrt{3}\,\gamma$ and $ a\geq \sqrt{3}\,\gamma$. Inequality \eref{univariate} will  follow from Inequalities \eref{res1}, \eref{res2}, \eref{res3} and \eref{res4}.

\subsubsection*{Case $0\leq a\leq 2\gamma^{-1}$}
From \eref{ineq1} we get
\begin{eqnarray*}
{\E\cro{\pa{a-{\exp\pa{\beta (a+\eps)^2}\over p\exp(\beta\lambda)+\exp\pa{\beta (a+\eps)^2}}(a+\eps)}^2}}
&\leq& 2a^2+2\E\cro{\pa{{\eps\over 1+ p\exp(\beta\cro{\lambda-(a+\eps)^2})}}^2}\\
&\leq& 2a^2+4\E\cro{\pa{{\eps{\bf 1}_{\eps>0}\over 1+ p\exp(\beta\cro{\lambda-(a+\eps)^2})}}^2}\\
&\leq&2a^2+ 4\E\cro{{\eps^2{\bf 1}_{\eps>0}\exp\pa{-2\beta\cro{\gamma^2-(a+\eps)^2}_{+}}}}
\end{eqnarray*}
with $\gamma^2=\lambda+\beta^{-1}\log p$. Expanding the expectation gives
\begin{eqnarray*}
\lefteqn{\E\cro{{\eps^2{\bf 1}_{\eps>0}\exp\pa{-2\beta\cro{\gamma^2-(a+\eps)^2}_{+}}}}}\\
&=& e^{-2\beta \gamma^2}\int_{0}^{\gamma-a}x^2e^{2\beta(a+x)^2-x^2/2}\,{dx \over \sqrt{2\pi}}+\int_{\gamma-a}^{\infty}x^2e^{-x^2/2}\,{dx \over \sqrt{2\pi}}
\end{eqnarray*}
Note that when $p\geq 3$, $\lambda\geq2$ and $\beta\leq 1/2$, we have
$\gamma\geq 2$. Therefore, when $0\leq a\leq 1$, an integration by parts in the second integral gives
\begin{eqnarray*}
\int_{\gamma-a}^{\infty}x^2e^{-x^2/2}\,{dx \over \sqrt{2\pi}}
&=& \cro{{-xe^{-x^2/2} \over \sqrt{2\pi}}}_{\gamma-a}^{\infty}+\int_{\gamma-a}^{\infty}e^{-x^2/2}\,{dx \over \sqrt{2\pi}}\\
&\leq& {2(\gamma-a)\over \sqrt{2\pi}}\,e^{-(\gamma-a)^2/2}.
\end{eqnarray*}
For the first integral, since $\beta> 1/4$, we have the  bound
$$e^{-2\beta \gamma^2}\int_{0}^{\gamma-a}x^2e^{2\beta(a+x)^2-x^2/2}\,{dx \over \sqrt{2\pi}}\leq e^{-(\gamma-a)^2/2}\int_{0}^{\gamma-a}x^2\,{dx \over \sqrt{2\pi}}\leq {(\gamma-a)^3\over 3\sqrt{2\pi}}\,e^{-(\gamma-a)^2/2}.$$
Besides an integration by parts also gives
\begin{eqnarray*}
\lefteqn{e^{-2\beta \gamma^2}\int_{0}^{\gamma-a}x^2e^{2\beta(a+x)^2-x^2/2}\,{dx \over \sqrt{2\pi}}}\\
&\leq&(4\beta-1)^{-1}e^{-2\beta \gamma^2}\int_{0}^{\gamma-a}(4\beta(a+x)-x)xe^{2\beta(a+x)^2-x^2/2}\,{dx \over \sqrt{2\pi}}\\
&\leq&{e^{-2\beta \gamma^2}\over (4\beta-1)\sqrt{2\pi}}\cro{xe^{2\beta(a+x)^2-x^2/2}}_{0}^{\gamma-a}={\gamma-a\over (4\beta-1)\sqrt{2\pi}}\,e^{-(\gamma-a)^2/2}.
\end{eqnarray*}
Putting pieces together, we obtain for $0\leq a\leq 1$ and $1/4< \beta\leq 1/2$
\begin{eqnarray}
\lefteqn{\E\cro{\pa{a-{\exp\pa{\beta (a+\eps)^2}\over p\exp(\beta\lambda)+\exp\pa{\beta (a+\eps)^2}}(a+\eps)}^2}}\nonumber\\
&\leq& 2a^2+{8+4\min\pa{(4\beta-1)^{-1},3^{-1}(\gamma-a)^2}\over \sqrt{2\pi}}(\gamma-a)e^{-(\gamma-a)^2/2}\nonumber\\
&\leq& 2a^2+{8+4\min\pa{(4\beta-1)^{-1},3^{-1}(\gamma-a)^2e^{-(1/2-\beta)(\gamma-a)^2}}\over \sqrt{2\pi}}(\gamma-a) e^{-\beta(\gamma-a)^2}\nonumber\\
&\leq& 2a^2+{8+4\min\pa{(4\beta-1)^{-1},\cro{3e(1/2-\beta)}^{-1}}\over \sqrt{2\pi}}\,\gamma e^{-\beta(\gamma-a)^2}\nonumber\\
&\leq& 2a^2+5.6\,\gamma\, e^{-\beta(\gamma-a)^2}\label{ineq2}.
\end{eqnarray}
Furthermore, when $0\leq a\leq 2\gamma^{-1}$, we have
 $(\gamma-a)^2\geq \gamma^2-4$. The inequalities \eref{ineq2} and 
 $\lambda \geq 2$  thus give
\begin{eqnarray}
\E\cro{\pa{a-{\exp\pa{\beta (a+\eps)^2}\over p\exp(\beta\lambda)+\exp\pa{\beta (a+\eps)^2}}(a+\eps)}^2}
&\leq& 2a^2+5.6e^{(4-\lambda)\beta}\gamma/p\nonumber\\
&\leq& 2a^2+8\,{\gamma^2/ p}\,.\label{res1}
\end{eqnarray}

\subsubsection*{Case $2\gamma^{-1}\leq a\leq 1$}
Starting from \eref{ineq2} we have for $2\gamma^{-1}\leq a\leq 1$
\begin{eqnarray}
\E\cro{\pa{a-{\exp\pa{\beta (a+\eps)^2}\over p\exp(\beta\lambda)+\exp\pa{\beta (a+\eps)^2}}(a+\eps)}^2}
&\leq& 2a^2+{5.6\,\gamma^3 e^{-\beta(\gamma-1)^2}\over \gamma^2}\nonumber\\
&\leq& 2a^2+{5.6\,\gamma^3 e^{-(\gamma-1)^2/4}\over \gamma^2}\nonumber\\
&\leq&2a^2+56\gamma^{-2}\ \leq\ 16a^2.\label{res2}
\end{eqnarray}

\subsubsection*{Case $1\leq a\leq \sqrt{3}\,\gamma $ }
From \eref{ineq1} we get 
\begin{eqnarray}
\E\cro{\pa{a-{\exp\pa{\beta (a+\eps)^2}\over p\exp(\beta\lambda)+\exp\pa{\beta (a+\eps)^2}}(a+\eps)}^2}
&\leq& 2(a^2+1)\nonumber\\
&\leq& 4a^2\ \leq\ 12\min(a^2,\gamma^2). \label{res3}
\end{eqnarray}

\subsubsection*{Case $a>\sqrt{3}\gamma$}\label{mugrand}
From \eref{ineq1} we get 
\begin{eqnarray*}
\lefteqn{\E\cro{\pa{a-{\exp\pa{\beta (a+\eps)^2}\over p\exp(\beta\lambda)+\exp\pa{\beta (a+\eps)^2}}(a+\eps)}^2}}\\
&\leq &4a^2\,\E\cro{\pa{{1\over 1+p^{-1}\exp\pa{\beta \cro{(a-\eps)^2-\lambda}}}}^2{\bf 1}_{\eps>0}}+2\\
&\leq& 4a^2\,\E\cro{\exp\pa{-2\beta \cro{(a-\eps)^2-\gamma^2}_{+}}{\bf 1}_{\eps>0}}+2\\
&\leq& 2a^2\exp\pa{-2\beta \cro{(2a/3)^2-\gamma^2}_{+}}+4a^2\p\pa{\eps>a/3}+2.
\end{eqnarray*}
Since $(2a/3)^2>a^2/9+\gamma^2$, we finally obtain
\begin{eqnarray}
\lefteqn{\E\cro{\pa{a-{\exp\pa{\beta (a+\eps)^2}\over p\exp(\beta\lambda)+\exp\pa{\beta (a+\eps)^2}}(a+\eps)}^2}}\nonumber\\
&\leq &2a^2\exp\pa{-2\beta a^2/9}+4a^2\exp\pa{-a^2/18}+2\nonumber\\
&\leq& 42\leq 11{\gamma^2}.\label{res4}
\end{eqnarray}

\subsection{Proof of Proposition \ref{resultatortho}}\label{preuveT2}
We can express the weights $c_{j}$ appearing in \eref{fasteq2} in the following way
$$c_{j}={e^{\hat\beta Z_{j}^2/\sigma^2}\over pe^{\hat\beta \hat\lambda}+e^{\hat\beta Z_{j}^2/\sigma^2}}$$
with $\hat \beta=\beta\sigma^2/\hat\sigma^2$ and $\hat\lambda=b/\hat\beta$. 
Note that $\hat\beta\leq 1/2$ enforces $\hat\lambda\geq 2$ since $b\geq1$.

Since $\hat\sigma^2$ is independent of the $Z_{j}$s, 
we can work conditionally on $\hat \sigma^2$. When $\hat\beta$ is not larger than $1/2$ we  apply Lemma \ref{resultatvarianceconnue2} and get
\begin{eqnarray}
\E\pa{\|\mu-\hmu\|^2\,|\,\hat\sigma^2}{\bf 1}_{\ac{\hat\beta\leq 1/2}}&\leq& 16\inf_{m\in\M}\cro{{\|\mu-\mu_{m}\|^2}+\pa{\hat\lambda+\hat\beta^{-1}\log p}(|m|+1)\sigma^2}{\bf 1}_{\ac{\hat\beta\leq 1/2}}\nonumber\\
&\leq&16\inf_{m\in\M}\cro{{\|\mu-\mu_{m}\|^2}+{1\over \beta}\pa{b+\log p}(|m|+1)\hat\sigma^2}{\bf 1}_{\ac{\hat\beta\leq 1/2}}\label{facile}.
\end{eqnarray}

When $\hat\beta$ is larger than $1/2$, we use the following bound.
\begin{lemma}
Write $Z$ for a Gaussian random variable with mean $a$ and variance $\sigma^2$. Then for any $\hat\lambda\geq 0$ and $\hat\beta>1/2$
\begin{equation}\label{residu}
\E\cro{\pa{a-{e^{\hat\beta Z^2/\sigma^2}\over pe^{\hat\beta \hat\lambda}+e^{\hat\beta Z^2/\sigma^2}}\, Z}^2}\leq 6 \pa{\hat \lambda+\hat\beta^{-1}\log p} \sigma^2+36\sigma^2.
\end{equation}
\end{lemma}
\begin{proof}
First, we write $\eps$ for a standard Gaussian random variable and obtain
\begin{eqnarray*}
\lefteqn{\E\cro{\pa{a-{e^{\hat\beta Z^2/\sigma^2}\over pe^{\hat\beta \hat\lambda}+e^{\hat\beta Z^2/\sigma^2}}\, Z}^2}}\\
&=&\E\cro{\pa{{a\over 1+p^{-1}\exp\pa{\hat\beta \cro{(a/\sigma+\eps)^2-\hat\lambda}}}  - {\sigma\eps\over 1+ p\exp\pa{\hat\beta\cro{\hat\lambda-(a/\sigma+\eps)^2}}}}^2}\\
&\leq&2(a^2+\sigma^2).
\end{eqnarray*}
Whenever $a^2$ is smaller than $3\pa{\hat \lambda+\hat\beta^{-1}\log p} \sigma^2$, this quantity remains smaller than $6\pa{\hat \lambda+\hat\beta^{-1}\log p} \sigma^2+2\sigma^2$. 

When $a^2$ is larger than $3\pa{\hat \lambda+\hat\beta^{-1}\log p} \sigma^2$
we follow the same lines as in the last case of Section \ref{mugrand} and get
\begin{multline*}
{\E\cro{\pa{a-{e^{\hat\beta Z^2/\sigma^2}\over pe^{\hat\beta \hat\lambda}+e^{\hat\beta Z^2/\sigma^2}}\, Z}^2}}\\
\leq 2\sigma^2+4a^2\p\pa{\eps>|a|/(3\sigma)}+2a^2p^2\exp\cro{-2\hat\beta\pa{\pa{2a\over 3\sigma}^2-\hat\lambda}}.
\end{multline*}
Since $(2a/3)^2\geq a^2/9+\pa{\hat\lambda+\hat\beta^{-1}\log p}\sigma^2$, we obtain
\begin{eqnarray*}
{\E\cro{\pa{a-{e^{\hat\beta Z^2/\sigma^2}\over pe^{\hat\beta \hat\lambda}+e^{\hat\beta Z^2/\sigma^2}}\, Z}^2}}
&\leq& 2\sigma^2+4a^2\exp\pa{-a^2/(18\sigma^2)}+2a^2\exp\pa{-a^2/(9\sigma^2)}\\
&\leq& 36 \sigma^2.
\end{eqnarray*}
\end{proof}
From \eref{residu}, we obtain after summation 
\begin{equation}\label{bof}
\E\pa{\|\mu-\hmu\|^2\,|\,\hat\sigma^2}\,{\bf 1}_{\ac{\hat\beta> 1/2}}\leq  p
\cro{{6\over \beta}\pa{b+\log p}\hat\sigma^2+36\sigma^2}{\bf 1}_{\ac{\hat\beta> 1/2}}+\|\mu-\Pi_{\S_{*}}\mu\|^2{\bf 1}_{\ac{\hat\beta> 1/2}}.
\end{equation}
Futhermore $\hat\sigma^2$ is smaller than $2\beta\sigma^2$ when $\hat\beta$ is larger than $1/2$, so 
taking the expectation of \eref{bof} and \eref{facile}  gives
\begin{eqnarray*}
{\E\pa{\|\mu-\hmu\|^2}}
&\leq&
16\inf_{m\in\M}\cro{{\|\mu-\mu_{m}\|^2}+{1\over \beta}\pa{b+\log p}(|m|+1)\bar\sigma^2}\\
&+&p\cro{12\pa{b+\log p}+36}\p\pa{\hat\sigma^2<2\beta\sigma^2}\sigma^2.
\end{eqnarray*}
The random variable $\hat\sigma^2/\sigma^2$ is stochastically larger than a random variable $X$ distributed as a $\chi^2$ of dimension $n-p$ divided by $n-p$. The Lemma~\ref{deviation} gives
$$\p\pa{{X}\leq {2\beta}}\leq{(2\beta)^{N/2}}\,\exp\pa{{N\over 2}(1-2\beta)}\leq \exp\pa{-N\phi(2\beta)}$$
and then 
$\p\pa{\hat\sigma^2<2\beta\sigma^2}\leq \exp(-(n-p)\phi(2\beta))$, so Condition \eref{conditionsortho} ensures that 
$$p\,\p\pa{\hat\sigma^2<2\beta\sigma^2}\leq 1.$$
Finally, since $12(b+\log p)+36$ is smaller than  $16(2+b+\log p)$ we get \eref{formulevarianceinconnue}.


\begin{thebibliography}{99}
\bibitem{A69} Akaike, H. (1969) {\it Statistical predictor identification.} Annals of the Institute of Statistical Mathematics. {\bf 22}, 203--217
\bibitem{BGH06} Baraud, Y.,  Giraud, C. and Huet, S. (2006) {\it Gaussian model selection with unknown variance.} http://arXiv:math/0701250v1
\bibitem{B87} Barron, A. (1987) {\it Are Bayesian rules consistent in information?} Open problems in Communication and Computation. T. Cover and B. Gopinath, Eds. Springer-Verlag.
\bibitem{BBM99} Barron, A., Birg\'e, L. and Massart, P. (1999) {\it Risk bounds for model selection via penalization.} Probab. Theory Related Fields {\bf 113} no. 3, 301--413.
\bibitem{BC91} Barron, A. and Cover, T. (1991)  {\it Minimum complexity density estimation.} IEEE Trans. Inform. Theory {\bf 37} no. 4, 1034--1054.
\bibitem{BM00} Birg\'e, L. and Massart, P. (2000) {\it An adaptive compression algorithm in Besov spaces.} Constr. Approx. {\bf 16}, no. 1, 1--36.
\bibitem{BM01} Birg\'e, L. and Massart, P. (2001) {\it Gaussian model selection.} J. Eur. Math. Soc. (JEMS)  {\bf 3}, no. 3, 203--268. 
\bibitem{BM06} Birg\'e, L. and Massart, P. (2006) {\it A generalized $C_{p}$ criterion for Gaussian model selection}. { To appear in Probab. Theory Related Fields}.
\bibitem{BTW07}  Bunea, F., Tsybakov, A. and Wegkamp, M. (2007) {\it Aggregation for Gaussian regression.} Ann. Statist. {\bf 35}, no. 4, 1674--1697.
\bibitem{C97} Catoni, O. (1997) {\it Mixture approach to universal model selection.} Laboratoire de l'Ecole Normale Sup\'erieure, Paris. Preprint 30.
\bibitem{C99} Catoni, O. (1999) {\it Universal aggregation rules with exact bias bounds.} Laboratoire de Probabilit\'es et Mod\`eles Al\'eatoires, CNRS, Paris. Preprint 510. 
\bibitem{DVL93} Devore, R. and Lorentz, G. (1993) {\it Constructive approximation.} Grundlehren der Mathematischen Wissenschaften [Fundamental Principles of Mathematical Sciences], {\bf 303}. Springer-Verlag
\bibitem{DJ94} Donoho, D. and Johnstone, I. (1994)  {\it Ideal spatial adaptation by wavelet shrinkage.} Biometrika {\bf 81}, no. 3, 425--455.
\bibitem{HKT90} Hall, P., Kay, J. and Titterington, D. M. (1990). {\it Asymptotically optimal differencebased
estimation of variance in nonparametric regression.} Biometrika {\bf 77}, 521--528.
\bibitem{H02} Hartigan, J.A. (2002) {\it Bayesian regression using Akaike priors.} Yale University, New Haven. Preprint.
\bibitem{L89} Lenth, R. (1989) {\it  Quick and easy analysis of unreplicated factorials.}
Technometrics  {\bf 31},  no. 4, 469--473.
\bibitem{LB06} Leung, G. and Barron, A. (2006) {\it Information Theory and Mixing Least-Squares Regressions.}  IEEE Transact. Inf. Theory {\bf 52}, no. 8, 3396--3410.
\bibitem{M73} Mallows, C. (1973) {\it Some comments on $C_{p}$.} Technometrics {\bf 15}, 661-675.
\bibitem{MBWF05} Munk, A., Bissantz, N., Wagner, T. and Freitag, G. (2005). {\it  On difference based variance
estimation in nonparametric regression when the covariate is high dimensional.}
J. Roy. Statist. Soc. B {\bf 67}, 19--41.
\bibitem{R84} Rice, J. (1984). {\it Bandwidth choice for nonparametric kernel regression.} Ann. Statist. {\bf 12}, 1215--1230.
\bibitem{S81} Stein, C. (1981) {\it Estimation of the mean of a multivariate normal distribution.}  Ann. Statist. {\bf 9}, 1135--1151.
\bibitem{TW05} Tong, T., Wang, Y. (2005) {\it Estimating residual variance in
nonparametric regression using least squares.}  Biometrika  {\bf 92},
no. 4, 821--830.
\bibitem{T03} Tsybakov, A. (2003) {\it Optimal rates of aggregation.} COLT-2003,  Lecture Notes in Artificial Intelligence, {\bf 2777}, Springer, Heidelberg, 303--313. 
\bibitem{YB99} Yang, Y. and Barron, A. (1999) {\it Information-theoretic determination of minimax rates of convergence} Ann. Statist. {\bf 27}, no. 5, 1564--1599.
\bibitem{Y00} Yang, Y. (2000)  {\it Combining different procedures for adaptive regression.} J. Multivariate Anal. {\bf 74}, no. 1, 135--161.
\bibitem{Y00b} Yang, Y.  (2000) {\it Mixing Strategies for Density Estimation.} Ann.  Statist. {\bf 28}, 75--87. 
\bibitem{Y03} Yang, Y. (2003) {\it Regression with multiple candidate models: selecting or mixing?} Statist. Sinica {\bf 13}, no. 3, 783--809.
\bibitem{Y04} Yang, Y. (2004)  {\it Combining forecasting procedures: some theoretical results.} Econometric Theory {\bf 20}, no. 1, 176--222
\bibitem{WBCL07} Wang, L., Brown, L., Cai, T. and Levine, M. {\it Effect of mean on variance function estimation in nonparametric regression.} To appear in Ann. Statist.
\end{thebibliography}
\end{document}